\documentclass[10pt]{amsart}

\usepackage{amsfonts,amsmath,latexsym,amssymb,verbatim,amsbsy,amsthm}
\usepackage{dsfont,bm}
\usepackage{multirow}
\usepackage{enumitem}
\usepackage{wrapfig}

\usepackage[top=1in, bottom=1in, left=1in, right=1in]{geometry}

\usepackage[dvipsnames]{xcolor} 

\usepackage{tikz}

\usepackage[colorlinks=true, pdfstartview=FitV, linkcolor=RoyalBlue, citecolor=ForestGreen, urlcolor=blue]{hyperref}

\theoremstyle{plain}
\newtheorem{THEOREM}{Theorem}[section]

\newtheorem{theorem}[THEOREM]{Theorem}

\newtheorem{lemma}[THEOREM]{Lemma}
\newtheorem{proposition}[THEOREM]{Proposition}

\theoremstyle{definition}

\newtheorem{definition}[THEOREM]{Definition}

\theoremstyle{remark}

\newtheorem{remark}[THEOREM]{Remark}

\newtheorem{example}[THEOREM]{Example}

\newcommand{\thm}[1]{Theorem~\ref{#1}}
\newcommand{\lem}[1]{Lemma~\ref{#1}}
\newcommand{\defin}[1]{Definition~\ref{#1}}
\newcommand{\rem}[1]{Remark~\ref{#1}}

\newcommand{\prop}[1]{Proposition~\ref{#1}}

\newcommand{\sect}[1]{Section~\ref{#1}}

\newcommand{\fig}[1]{Figure~\ref{#1}}

\newcommand{\R}{\ensuremath{\mathbb{R}}}   
\newcommand{\T}{\ensuremath{\mathbb{T}}}   

\def \a {\alpha}
\def \b {\beta}
\def \g {\gamma}
\def \d {\delta}
\def \e {\varepsilon}
\def \f {\varphi}

\def \l {\lambda}
\def \n {\nabla}

\def \t {\tau}
\def \th {\theta}

\def \w {\omega}

\def \L {\Lambda}
\def \Th {\Theta}
\def \O {\Omega}

\def \ba {{\bf a}}

\def \bc {{\bf c}}

\def \cA {\mathcal{A}}

\def \cC {\mathcal{C}}

\def \cF {\mathcal{F}}

\def \cK {\mathcal{K}}

\def \cLF {\mathcal{LF}}
\def \cM {\mathcal{M}}

\def \cO {\mathcal{O}}
\def \cP {\mathcal{P}}

\def \orho {{\bar{\rho}}}

\def \oH { \overline{H} }

\def \oS { \overline{S} }

\def \oTh {\overline{\Th}}

\def \oTh { \overline{\Th} }

\def \loc {\mathrm{loc}}

\def \one {{\mathds{1}}}
\def \Lip {\mathrm{Lip}}

\newcommand{\jap}[1]{\left\langle #1 \right\rangle}

\DeclareMathOperator{\supp}{supp} %

\def \lan {\langle}
\def \ran {\rangle}

\def \p {\partial}

\def \ss {\subset}

\def \GL {Gr\"onwall's Lemma}
\def \HI {the H\"older inequality}

\renewcommand{\geq}{\geqslant}

\renewcommand{\leq}{\leqslant}

\def \dv  {\, \mbox{d}v}
\def \dx  {\, \mbox{d}x}

\def \dt  {\, \mbox{d}t}

\def \dy  {\, \mbox{d}y}
\def \dz  {\, \mbox{d}z}

\def \dr  {\, \mbox{d}r}
\def \ds  {\, \mbox{d}s}

\def \dmu  {\, \mbox{d}\mu}

\def \dnu  {\, \mbox{d}\nu}

\def \deta  {\, \mbox{d}\eta}

\def \dtau  {\, \mbox{d}\tau}
\def \drho  {\, \mbox{d}\rho}

\def \ddt  {\frac{\mbox{d\,\,}}{\mbox{d}t}}

\def \dSigma  {\, \mbox{d}\Sigma}

\def \domain {{\O \times \R^n}}
\def \ddomain {{\O^2 \times \R^{2n}}}
\def \dddomain {{\O^3 \times \R^{3n}}}

\def \st {\mathrm{s}}

\def \TV {\mathrm{TV}}

\begin{document}

\title[]{On kinetic alignment models with rough communication kernel}

\author{Roman Shvydkoy}
\email{shvydkoy@uic.edu}

\address{Department of Mathematics, Statistics and Computer Science, University of Illinois at Chicago, Chicago, IL 60607, USA}

\subjclass[2020]{35Q83, 35Q84}

\date{\today}

\keywords{}

\thanks{\textbf{Acknowledgment.}  
The work of RS  was  supported in part by NSF
grant DMS-2405326 and Simons Foundation}

\maketitle

\begin{abstract}
We prove the weak maximum principle and the Cucker-Smale flocking theorem for a general class of  kinetic alignment models  with rough communication kernels. This includes the classical Cucker-Smale and Motsch-Tadmor systems with defining kernel $\phi$ assumed to be merely continuous or even $\phi\in L^\infty$ for solutions in the Lebesgue class $f\in L^1$.  For such models the characteristic flow-map is not well-defined and therefore the classical approaches of Cucker-Smale \cite{CS2007a,CS2007b} or Ha-Liu-Tadmor \cite{HL2009,HT2008} break down. We develop a new method adopted to rough settings  based on Ambrossio's probabilistic representation of solutions and construction of dynamic barriers.  All proofs are performed in the framework of so called bounded communication protocols, a non-smooth variant of regular protocols introduced in \cite{Shv-CPAM, Shv-EA} and representing a broad range of models of collective dynamics.
\end{abstract}




\section{Introduction}

In collective dynamics it is widely believed, see for example \cite{VZ2012} that fine features of communication between agents are not important and in fact may not even be precisely determined. Instead, it is macro features such as locality and the strength of long range interactions that play decisive role in collective outcome.

In the context of the classical Cucker-Smale system introduced in \cite{CS2007a,CS2007b}
\begin{equation}\label{e:CS}
\dot{x}_i = v_i, \quad \dot{v}_i =\sum_{j=1}^N m_j \phi(x_i-x_j) (v_j - v_i),
\end{equation}
this thesis was corroborated by Ha and Tadmor \cite{HT2008} and Ha and Liu \cite{HL2009} by extending the famous Cucker-Smale Flocking Theorem from the class of precise power-law kernels\footnote{We use the Japanese brackets notation $\jap{x} = (1+|x|^2)^{1/2}$.} $\phi(x-y) = \frac{H}{\jap{x-y}^\b}$, $\b \leq 1$ to general radially symmetric kernels with heavy tail:  $\int_0^\infty \phi(r)\dr = \infty$. In fact, assuming $\phi>0$ is smooth, fat-tail, and decreasing, every solution to \eqref{e:CS} aligns exponentially fast to the mean velocity $\bar{v} =\frac{1}{N} \sum_{j=1}^N v_j$, while the flock remains bounded  
\[
\max_{i=1,\ldots,N} |v_i - \bar{v}| \leq C e^{-\d t}, \qquad \max_{i,j=1,\ldots,N} |x_i - x_j| \leq \bar{D}.
\]

The Cucker-Smale theorem, later extended to smooth solutions of the corresponding kinetic \cite{CFRT2010,CFTV2010} and hydrodynamic models \cite{TT2014}, relies crucially on the well-posedness of the corresponding characteristic flow-map in the Hadamard sense -- the characteristics \eqref{e:CS} (or their kinetic and hydrodynamic counterparts) need to be global, unique and continuously dependent on the initial data to enable application of the maximum principle on the velocity components and construction of the Lyapunov function as in \cite{HL2009}, or weighted Fiedler number as in  \cite{CS2007a}. Thus, so far the theory targets models based primarily on smooth communication kernels -- a fine feature that leaves out natural examples such as $\phi(x) = e^{-\sqrt{|x|}}$ for $|x|<1$, and any bounded kernel with jump or cusp discontinuity. Yet such kernels are widely used in opinion dynamics for example \cite{HKmodel,MT2014} and other applications.

 Another cornerstone example of communication adapted to the study of heterogeneous flocks is given by the Motsch-Tadmor model \cite{MT2011}:
\begin{equation}\label{e:MT}
\dot{x}_i = v_i, \quad \dot{v}_i =\frac{1}{\sum_{k=1}^N m_k \phi(x_i-x_k)} \sum_{j=1}^N m_j \phi(x_i-x_j) (v_j - v_i),
\end{equation}
Here, the kernel is given by $\phi_\rho(x,y) =  \frac{\phi(x-y)}{\rho \ast \phi(x)}$. If the defining kernel is locally supported, $\phi(x) \geq \l \one_{|x|<r_0}$ this density-dependent communication makes the system \eqref{e:MT} well-posed provided $\phi \in \Lip$, and the Cucker-Smale flocking theorem holds under the same heavy tail condition, see \cite{MT2014}. However, if $\phi$ is below $\Lip$ or for $L^1$-densities in the kinetic context, $\phi_\rho$ may loose regularity and can even become unbounded due to vacuum.

The purpose of this paper is to carry out the classical flocking theory in the context of models with rough communication kernels. We consider a more inclusive description given by the kinetic Vlasov-alignment equation:
\begin{equation}\label{e:VA}
\p_t f + v \cdot \n_x f + \n_v \cdot ( \cA(f) f ) =0.
\end{equation}
The equation is considered on the phase space $\O \times \R^n$, where $\O$ represents a physical environment assumed to be either the periodic box $\T^n$ or open space $\R^n$, and an alignment force  $\cA(f)$ given by
\begin{equation}\label{ }
\cA(f)(x,v) = \int_{\domain} \phi_\rho(x,x') (v' - v) f(x',v') \dv' \dx',
\end{equation}
where $\phi_\rho(x,y) \geq 0$ is a communication kernel of Schur class (see \ref{i:f2} - \ref{i:f3} below) but without any regularity assumptions. Examples of such kernels include Cucker-Smale with merely continuous $\phi \in C(\O)$ or even $\phi \in L^\infty(\O)$ if densities are  Lebesgue  $\rho\in L^1$, and Motsch-Tadmor models under the same assumption on $\phi$. Many other examples will be discussed in \sect{s:applications} and further literature can be found in these surveys \cite{ABFHKPPS,Axel97,Ben2005,Edel2001,Darwin,VZ2012,MT2014,MP2018,Shv-book,Tadmor-notices}. 

We consider \eqref{e:VA} in the context of measure-values weak solutions $f = \mu \in \cP_1(\O\times \R^n)$, in the space of probability measures with bounded first moment, which include discrete systems \eqref{e:CS} and \eqref{e:MT} as a particular realization of \eqref{e:VA} for empirical solutions
\begin{equation}\label{e:empirical}
\mu^N = \sum_{i=1}^N m_i \d_{x_i(t)} \otimes \d_{v_i(t)}.
\end{equation}

In smooth settings justification of \eqref{e:VA} via the mean-field limit $\mu^N \to f \dv\dx$ and Dobrushin stability argument was done for various models in \cite{HL2009,HT2008,MP2018,Shv-book} and in the general framework of regular communication protocols in \cite{Shv-EA}.

We will establish the following general results (see \prop{p:mp} and \prop{p:CS} for precise statements):

$\bullet$ (Weak Maximum Principle) If $\mu \in L^\infty([0,T); \cP_1(\O \times \R^n))$ is a weak solution to \eqref{e:VA} with a compact initial support $\supp \mu_0 \ss \cO \times \cK$, where $\cK \ss \R^n$ is convex and bounded, then  for all $t>0$,
\begin{equation}\label{e:mp}
\supp \mu(t) \ss (\cO +t \cK) \times \cK.
\end{equation}
In particular, the support of the solution in $v$-direction will remain confined to its initial convex hull.
 
$\bullet$ (Cucker-Smale Flocking Theorem) Suppose the kernel has a power lower bound with heavy tail:
\begin{equation}\label{e:kerbelowintro}
\phi_\rho(x,y) \geq \frac{\l}{\jap{x-y}^\b},
\end{equation}
for some $\b \leq 1$ and $\l>0$ and all $\rho\in \cP(\O)$.
Let $\mu \in L^\infty([0,T); \cP_1(\O \times \R^n))$ be any weak solution to \eqref{e:VA} with a compact initial support $\supp \mu_0 \ss \cO \times \cK$. Then the diameter of the flock remains bounded
\[
D(t): =  \sup\{ |x-y|: \ x,y \in \supp \rho(t)\} <\infty,
\]
and exponential alignment holds
\[
A(t): = \sup\{ |v-w|: \  v,w \in \supp_v \mu(t) \} \leq A(0) e^{-\d t}.
\]

Because of the  lack of access to characteristics, our approach to prove these results will be radically different from traditional one. Instead, we adopt  Ambrossio's probabilistic interpretation of the continuity equation \eqref{e:VA}, see \sect{s:transport}, and a construction of dynamic barriers, which are special test-functions that surround the $v$-support of the solution and  ``squeeze" it at an exponential rate, see \sect{s:CSFT}.

In terms of existence of weak solutions, on the discrete level both Cucker-Smale \eqref{e:CS} and Motsch-Tadmor \eqref{e:MT} systems will have a solution from any initial data $\{x_i^0, v_i^0\}_i$ by the Peano Theorem provided the defining kernel $\phi$ is locally supported and continuous. Of course such solutions are not unique, but each one will obey both the maximum principle and the flocking results.
On the kinetic level existence of weak solutions was proved by Karper-Mellet-Trivisa \cite{KMT2013} for Cucker-Smale system with smooth kernel $\phi$ and for Motsch-Tadmor system for smooth and compactly supported kernels:
\begin{equation}\label{e:KMTrRintro}
\l \one_{|x|\leq r} \leq \phi(x) \leq \L \one_{|x|\leq R}.
\end{equation}
This assumption was necessary to ensure that the communication operator $\bc_\rho(u)(x) = \int_\O \phi_\rho(x,y) u(y) \drho(y)$, which is an integral component of the alignment force $\cA(f)$ is bounded on $L^2(\rho)$.  Such restriction on the support excludes heavy-tailed kernels and leaves the model outside the scope of the Cucker-Smale flocking theorem.  We will fill this gap by the following result.

$\bullet$ (Existence of Weak Solutions) If the family of kernels $\phi_\rho$, $\rho\in L^1$, defines a so-called bounded communication protocol, see \defin{d:protocol}, then for any initial compactly supported distribution $f_0\in L^p(B_R \times B_R)$, for $p>1$, there exists a global weak solution $f \in L^\infty([0,T); L^p(B_{R+TR} \times B_R))$ for any $T>0$. See \thm{t:weak} for a precise statement.

We will show that for Motsch-Tadmor and similar singular models the definition of a bounded protocol applies not only to kernels satisfying \eqref{e:KMTrRintro} but also to rough kernels $\phi \in L^\infty$ satisfying power bounds
\begin{equation}\label{e:powerlikeintro}
\frac{\l}{\jap{x}^\b} \leq \phi(x) \leq \frac{\L}{\jap{x}^\b}, \quad \b>0.
\end{equation}
Since $\rho \ast \phi(x) \leq \|\phi\|_\infty$, we have \eqref{e:kerbelowintro} satisfied for such kernels, and consequently, the Cucker-Smale theorem applies if $\b\leq 1$.

In \sect{s:applications} we will present a broad collection of models that fit into the framework of bounded protocols, and hence, to which all the bulleted statements above apply.

\subsection{Notation}
We denote by $\O$ either the periodic domain $\T^n$ or the open space environment $\R^n$ or the direct product of the two. Let $\cM(\O)$ denote the space of Borel measures on $\O$ with bounded total variation we denote $\| \mu \|_\TV$, and $\cP(\O)\ss \cM(\O)$ denote the set probability measures on $\O$. If $\O$ is unbounded, for $q \geq 0$ we denote
\[
\cP_q(\O) =\left\{ \rho \in \cP(\O): \int_\O \jap{x}^q \drho(x) = \| \rho\|_{\cP_q} <\infty \right\}.
\]
On phase space $\domain$ we utilize similar definitions with equal weight in $x$ and $v$:
\begin{equation*}\label{}
\begin{split}
\cP_q(\domain) & =\left\{ \mu \in \cP(\domain): \int_\domain (\jap{x} + \jap{v})^q \dmu = \| \mu \|_{\cP_q} <\infty \right\}, \\
L^p_q(\domain) & =\left\{ f \in L^p(\domain): \int_\domain (\jap{x} + \jap{v})^q  f^p(x,v) \dv \dx = \| f \|_{L^p_q}^p <\infty \right\}, \\
\end{split}
\end{equation*}
 Further, we use the notation $C_b^k(\domain)$ to denote the class of bounded continuously $k$-times differentiable functions,  and $C^0_b = C_b$ for short.  We also denote by $L^p(\rho)$ the Lebesgue space relative to a measure $\rho\in \cP(\O)$.

\section{Bounded communication protocols}
\subsection{Kernel, strength, thickness} \label{s:EA} Informally, a {\em flock} is a probability measure $\rho \in \cP(\O)$ often with special properties, such as densities with compact support or finite higher moment $\rho \in \cP_q(\O)$, or given by an integrable density $\rho \in L^1_q(\O)$,  which reflects all the properties of a flock required in the context of a given model. Formally, we define a flock as a member of a subclass of admissible measures given by $\cF(\O) = \cP(\O) \cap \cLF(\O)$ where $\cLF(\O) \ss \cM(\O)$ is a linear subspace of finite Borel measures endowed with a norm $\| \cdot\|_{\cLF} \geq c \| \cdot \|_{\cM}$. We assume that $\cLF(\O)$ is rich enough to include at least compactly supported flocks in $L^1$: $L^1(B_R) \ss \cLF(\O)$ for all $R$, with the embedding constant possibly depending on $R$.  If $\rho \in \cF$ we simply denote $\|\rho \|_\cF = \| \rho \|_{\cLF}$.

By {\em communication kernel} of a flocking model we understand a family of non-negative Borel functions $\phi_\rho: \O \times \O \to [0,\infty]$, defined for all $\rho \in \cF(\O)$. At the very least the kernels are assumed to be integrable, $\phi_\rho \in L^1(\rho \otimes \rho)$.  However, for a proper analysis of differential models where these kernels are to be used,  more stringent integrability conditions are needed.  We introduce them next.

We first define the marginals of $\phi_\rho$, called {\em communication strength-functions}. Strength is a measure of  total influence of the flock on an agent $x\in \O$ and is defined by the $x$-marginal: 
\begin{equation}\label{e:s}
\st_\rho(x) = \int_\O \phi_\rho(x,y) \drho(y),
\end{equation}
while the adjoint strength is a measure of influence of an agent on the flock, and is defined by the $y$-marginal:
\begin{equation}\label{e:sadj}
\st^*_\rho(y) = \int_\O \phi_\rho(x,y) \drho(x).
\end{equation}

We will impose uniform boundedness of these strengths along with Hilbert-Schmidt condition:  for all $\rho \in \cF(\O)$, there are constants $\oH, \oS, \oS^*> 0$ which depend only on the norm $\| \rho \|_{\cF}$, such that the following bounds hold:
\newlist{kernel}{enumerate}{1}
\setlist[kernel, 1]
{label=($\phi$\arabic{kerneli}), 
leftmargin=30pt,
}

\begin{kernel}
\item \label{i:f1} $\int_{\O \times \O} \phi^2_\rho(x,y) \drho(y) \drho(x)  \leq \oH( \| \rho \|_{\cF})$,
\item \label{i:f2} $\sup_{x\in \O} \st_\rho(x)   \leq \oS(\| \rho \|_{\cF})$,
 \item  \label{i:f3}  $\sup_{y \in \O} \st^*_\rho(y)  \leq  \oS^*(\| \rho \|_{\cF})$.
\end{kernel}
In other words, the kernels remain uniformly Schur and Hilbert-Schmidt on bounded families of flocks in $ \cF$. On compact domains, we require no finite moments and we often set $\cF = \cP$ or $\cF = L^1$. In these cases, since $\| \rho \|_{\cF} = 1$, the constants $\oH, \oS, \oS^*$ are   independent of $\rho$.

Conditions \ref{i:f1} - \ref{i:f3} can be interpreted as variants of the same scaling assumption on the kernel as a function of  density: 
\[
\phi_\rho \lesssim \frac{1}{\rho}.
\]   

Let us  define communication operators
\begin{equation}\label{e:warep}
\bc_\rho(u)(x) = \int_\O \phi_\rho(x,y) u(y) \drho(y),
\end{equation}
and alignment operators, 
\begin{equation}\label{e:amacro}
\ba_\rho(u)(x) = \int_\O \phi_\rho(x,y) (u(y) - u(x)) \drho(y) = \bc_\rho(u) -  \st_\rho u.
\end{equation}
By the classical Schur test,  such operators are bounded on any $L^p(\rho)$, $1\leq p \leq \infty$, uniformly on bounded families of densities  $\rho \in \cF(\O)$:
\begin{equation}\label{ }
 \| \ba_\rho\|_{L^p(\rho) \to L^p(\rho)}  + \|\bc_\rho\|_{L^p(\rho) \to L^p(\rho)} \leq C(\oS,\oS^*).
\end{equation}

We isolate two important classes of kernels:
\begin{align}
\st_\rho &= \st^*_\rho, \quad \rho\text{-almost everywhere},  & & \text{ (conservative) }  \label{e:consdef}\\
\phi_\rho(x,y) &= \phi_\rho(y,x), \quad \rho \otimes \rho \text{-almost everywhere,}  & &\text{ (symmetric) } . \label{e:symm}
\end{align}
Clearly, every symmetric kernel is conservative. There are, however, examples  of non-symmetric conservative kernels, see \cite{Shv-EA}.

Another component of a flocking model is {\em thickness}. Intuitively, thickness as a measure of the mass of the flock around two communicating agents $(x,y)$.  We will define thickness by recording only the minimal set of requirements to encompass models with rough defining kernels $\phi$, yet sufficient for the construction of weak solutions to the Vlasov-alignment equations. So, our definition will be weaker that the one introduced in \cite{Shv-CPAM}.

\newlist{thickness}{enumerate}{1}

\setlist[thickness, 1]
{label=($\th$\arabic{thicknessi}), 
leftmargin=30pt,
}

\begin{definition}\label{d:th}
{\em Thickness} is a uniformly bounded family of lower semi-continuous functions $\th_\rho : \O\times \O \times \cF(\O) \to [0,\oTh]$ satisfying the following conditions:
\begin{thickness}
\item \label{i:th1}  $\rho \otimes \rho ( \{\th_\rho = 0 \} ) = 0$, for all $\rho \in \cF(\O)$;
\item  \label{i:th2} for every $R>0$ there exists a $R'>0$ such that if $m(R',\rho): = \inf_{|z| <R'} \rho(z) >0$, then 
\[
\inf_{|x|<R, |y| <R} \th_\rho(x,y) \geq c(m(R',\rho), R', R) >0;
\] 
\item  \label{i:th3} there exists a $C>0$ such that for all $\rho',\rho''\in \cF(\O)$,
\[
\sup_{\O \times \O} | \th_{\rho'} (x,y) - \th_{\rho''}(x,y)| \leq C\|\rho'-\rho''\|_{\TV}.
\]
\end{thickness}
\end{definition}
If $\O$ is a compact environment, \ref{i:th2} simply means that thickness is positive $\inf \th_\rho >0$ for non-vacuous densities $\inf \rho >0$.

\begin{definition}\label{d:protocol}
We say that a family $\cC(\O) = \{ \phi_\rho, \th_\rho\}_{\rho \in \cF(\O)}$ defines a {\em bounded communication protocol} on $\O$ if the kernel satisfies  \ref{i:f1} - \ref{i:f3}, the thickness satisfies \ref{i:th1} - \ref{i:th3}, and 
\newlist{regularity}{enumerate}{1}

\setlist[regularity, 1]
{label=(b\arabic{regularityi}), 
leftmargin=30pt,
}

\begin{regularity}
\item \label{i:b1}  there is a positive function $C: \R^2_+ \to \R_+$ such that for all $\rho\in \cF(\O)$ 
\[
 | \phi_\rho(x,y)|   \leq C(\th_\rho(x,y), \| \rho \|_{\cF}), \ \forall (x,y)\in \O \times \O,
\]
\item \label{i:b2} the following continuity-in-$\rho$ holds: for any  $\rho',\rho'' \in \cF(\O)$  we have
\[
|\phi_{\rho'} (x,y) -\phi_{\rho''}(x,y)| \leq C(\th_{\rho'}(x,y),\th_{\rho''}(x,y),\| \rho' \|_{\cF},\| \rho'' \|_{\cF})\| \rho' -  \rho''\|_{\TV},\ \forall (x,y)\in \O \times \O,
\]
for some $C: \R_+^4 \to \R_+$.
\end{regularity}

\end{definition}

\subsection{Classical examples} \label{s:examples}

Let us highlight a few classical examples of communication protocols. A much broader discussion will be postponed till \sect{s:examples}.

\subsubsection{Cucker-Smale} The Cucker-Smale model was originally introduced with a specific smooth power-law kernel
\begin{equation}\label{CS-orig}
\phi_\rho(x,y) = \frac{\l}{\jap{x-y}^\b}.
\end{equation}
Since it is smooth, the role of thickness related to boundedness becomes irrelevant, so we can set $ \th_\rho \equiv 1$. The set of admissible flocks is unrestricted $\cF = \cP(\O)$ for any choice of $\O$.

All the same conclusions apply for all bounded continuous convolution kernels $\phi\in C_b(\O)$,
\begin{equation}\label{CS}
\phi_\rho(x,y) = \phi(x-y).
\tag{CS}
\end{equation}

In the context of rougher kernels $\phi \in L^\infty(\O)$, we set $\cF =  L^1 \cap \cP$, so that the pairing between $\phi$ and density $\rho$ makes functional sense. 
The model is clearly symmetric with the communication  given by
\begin{equation}\label{ }
\bc_\rho(u) = (u \rho)\ast \phi, \quad \st_\rho = \rho \ast \phi.
\end{equation}

\subsubsection{Motsch-Tadmor} 
The Motsch-Tadmor model introduced in \cite{MT2011,MT2014} corresponds to the averaged version of the 
\ref{CS}-protocol\footnote{From here on we adopt a convention  that quotients involving density in the denominator are zero outside the support of the denominator.}, 
\begin{equation}\label{MT}
\begin{split}
\phi_\rho(x,y) = \frac{\phi(x-y)}{\rho \ast \phi(x)}, &\quad \th_\rho(x,y) = \rho \ast \phi(x). \\
\bc_\rho(u) = \frac{(u \rho)\ast \phi}{ \rho \ast \phi}, &\quad \st_\rho = \one_{\rho \ast \phi > 0}, \quad \st_\rho^* = \left(  \frac{\rho }{\rho \ast \phi }\right) \ast \phi.
\end{split}\tag{MT}
\end{equation}
Here, the class of admissible flocks depends on the domain $\O$. If $\O$ is compact, we can define $\cF$ exactly as for \ref{CS}-protocol: $\cF =  \cP$ if $\phi \in C_b(\O)$, and $\cF =  L^1 \cap \cP$ if $\phi \in L^\infty(\O)$.  If $\O$ is non-compact, the  kernel $\psi$ is required to be power-like: 
\begin{equation}\label{e:powerlike}
\frac{\l}{\jap{x}^\b} \leq \phi(x) \leq \frac{\L}{\jap{x}^\b},
\end{equation}
while the admissibility class is given by $\cF = L^1_q \cap \cP$ for $q\geq \b$. 

Verification of the axioms of a protocol is somewhat non-trivial for this models and will be postponed till \sect{s:examples} in the context of a much more general class of models.

\subsubsection{Adjoint Motsch-Tadmor}  Considering adjoint communication operator $\bc_\rho^*$ relative to the density-weighted inner product leads us to consider the adjoint protocol:
\begin{equation}\label{MT*}
\phi_\rho(x,y) = \frac{\phi(x-y)}{\rho \ast \phi(y)}, \quad \th_\rho(x,y) = \rho \ast \phi(y) . \tag{MT$^*$}
\end{equation}
Here the formulas for strength-functions switch, and the actual communication operator is given by
\begin{equation}\label{ }
\bc_\rho(u) = \left( \frac{u \rho}{ \rho \ast \phi} \right)\ast \phi.
\end{equation}

Again, \ref{MT*} is a particular case of a much broader class of $\g$-protocols to be discussed in \sect{s:examples}.

\section{Kinetic dynamics: Vlasov-Alignment equation}

Here and throughout we assume that $\cC$ is a bounded communication protocol over an admissible class of flocks $\cF$, although only the Schur assumptions \ref{i:f2} - \ref{i:f3} will be necessary for the results of this section.  
The corresponding Vlasov-Alignment equation  reads
\begin{equation}\label{e:VAmu}
\p_t \mu + v \cdot \n_x \mu +  \n_v \cdot (  \cA(\mu) \mu) = 0, 
\end{equation}
where
\[
\cA(\mu)(x,v) = \int_\domain \phi_\rho(x,x') (v'-v) \dmu(x',v').
\]

\begin{definition}\label{d:weakVA} We say that $\mu \in L^\infty([0,T); \cP_q(\O \times \R^n))$, $q\geq 1$,  is an admissible weak solution to  \eqref{e:VAmu} with initial condition $\mu_0$ if the macroscopic densities are admissible, $\rho \in L^\infty([0,T); \cF(\O))$, and for any test-function $g\in C_b^1([0,T)\times \O \times \R^{n})$ one has, for all $0<t<T$,
	\begin{equation}\label{e:weakVA}
	\begin{split}
\int_{\domain} g(t,x,v)\dmu(t,x,v)  &=	\int_{\domain} g(0,x,v)\dmu_0(x,v) \\
&+ \int_0^t \int_{\domain}  ( \p_s g + v\cdot \n_x  g + \cA(\mu) \cdot \n_v g)\dmu(s,x,v)\ds .
\end{split}
\end{equation}
\end{definition}
\begin{remark}($\cA(\mu) \in L^q(\mu)$) \label{r:P1} Since $\int_\domain \phi_\rho(x,x') \dmu(x',v') = \st_\rho(x) \in L^\infty(\rho)$ we have 
\[
|\cA(\mu)|^q \leq C \int_\domain \phi_\rho(x,x') ( |v|^q+|v'|^q)\dmu(x',v') , \quad \mu-\text{a.e.}
\]
By the Fubini-Tonelli  Theorem,
\begin{equation*}\label{}
\begin{split}
\int_\domain |\cA(\mu)|^q \dmu & \leq  \int_\domain \int_\domain \phi_\rho(x,x') (|v|^q + |v'|^q) \dmu(s,x',v')\dmu(s,x,v)  \\
&= \int_\domain \st_\rho(x) |v|^q \dmu(s,x,v) +  \int_\domain \st^*_\rho(x') |v'|^q \dmu(s,x',v') \leq ( \oS + \oS^*)  \|\mu(s)\|_{\cP_q} < C.
\end{split}
\end{equation*}
This bound ensures that the right hand side of \eqref{e:weakVA} makes functional sense.
\end{remark}

\begin{remark}(Test-functions) \label{r:ext1}
The definition \defin{d:weakVA} extends to include test-functions $g \in C^1_\loc([0,T)\times \O \times \R^{n})$ such that  $\p_t g, \p_x g, \p_v g \in L^\infty([0,T)\times \O \times \R^{n})$ and  $|g(t,x,v)| \leq C \jap{v}$.

The proof is elementary: consider a smooth cut-off function $\w_R = \w(v/R)$ for $R>0$ large, and denote $g_R = g \w_R$. All the terms converge by the Lebesgue Dominated Convergence Theorem, while the remainder term involving $\n_v \w_R = \frac1R \n_v \w(v/R)$ converges to zero in view of \rem{r:P1}, and the fact that $g\n_v \w_R$ is uniformly bounded and supported on $R/2 < |v|<R$.
\end{remark}

\begin{remark}(Narrow continuity) \label{} Following Ambrossio et al \cite{Ambrosio} we say that a map $t \to \mu(t)\in \cP$ is narrowly continuous on $[0,T)$, if 
for any $\f \in C_b(\domain)$, the function $\lan \f, \mu(t) \ran$ is continuous on $[0,T)$. If a family $\mu$ belongs to $L^\infty([0,T); \cP_q(\domain))$ for some $q>0$, then such family is automatically tight, and hence narrow continuity follows from the usual weak$^*$-continuity in the space of bounded measures $\cM(\domain)$. 

As a consequence of the previous remark we conclude that the map $t\to \lan g, \mu(t) \ran$ is continuous (even Lipschitz) for any $g\in C^1_b(\domain)$, and by approximation and tightness for any $g\in C_b(\domain)$. Hence, any weak solution $\mu$ satisfying \defin{d:weakVA} is automatically narrowly continuous in the sense of  \cite{Ambrosio}.  
\end{remark}
\begin{remark}(Absolute continuity) \label{}
Equation \eqref{e:VA} can be interpreted as a continuity equation 
\begin{equation*}\label{}
\begin{split}
\p_t \mu + \n_{x,v}( U \mu)&=0,\\
U(t, x,v) &= \lan v, \bc_\rho(u)(x) - \st_\rho(x) v \ran \in L^\infty([0,T); L^q(\mu)).
\end{split}
\end{equation*}
Consequently, see \cite{Ambrosio}, the curve $t \to \mu(t)$ is absolutely continuous in $\cP_p$ endowed with the Wasserstein metric $W_q$, and $|\mu'(t)| \leq \| U(t) \|_{L^q(\mu(t))}$. In particular, the solution is Lipschitz continuous
\[
W_q(\mu(t),\mu(s)) \leq C|t-s|.
\]
\end{remark}
\begin{remark}(Macroscopic quantities) \label{r:macro}
Let $\mu \in \cP_q$, for $q \geq 1$. Define  the macroscopic density and momentum  by
\begin{equation}\label{ }
\rho(\cdot) = \int_{v\in \R^n}  \mu(\cdot,v), \quad j(\cdot) = \int_{v\in \R^n} v \mu(\cdot,v).
\end{equation}
Since the first moment of $\mu$ is bounded, $j$ defines a measure on $\O$.  This measure is absolutely continuous with respect to $\rho$, $j \ll \rho$, and by the Radon-Nikodym Theorem there exists a field $u \in L^1(\rho)$ such that $j = u\rho$.  

Moreover, if $1 \leq q <\infty$, then for any $w\in L^p(\rho)$ with $\frac1p + \frac1q =1$ we obtain by \HI,
\[
\int_\O w u \drho = \int_\domain w(x) v \dmu \leq \left( \int_\domain |w(x)|^p \dmu\right)^{1/p} \left( \int_\domain |v|^q \dmu\right)^{1/q} \leq \|\mu\|_{\cP_q} \| w\|_{L^p(\rho)}.
\]
So,
\begin{equation}\label{e:umuLp}
\| u\|_{L^q(\rho)} \leq \|\mu\|_{\cP_q}.
\end{equation}

The case $q=\infty$ can be interpreted as $\mu$ having  a compact support in $v$. If $\supp_v \mu$ is confined to a convex set $\cK \ss \R^n$, then we can conclude also that $u\in \cK$, $\rho$-almost everywhere. Indeed, 
 $\cK$ can be represented as an intersection of hyperspaces:
\[
\cK = \bigcap_{\ell \in F \ss \R^n} \{v:  \ell(v) \leq c_\ell\},
\]
for some compact $F \ss \R^n$.  So, for any $\ell \in F$, $\ell( j ) = \int_{v\in \R^n} \ell(v) \mu(\cdot,v) \leq c_\ell \rho$. The Radon-Nikodym Theorem implies that the corresponding density $\ell(u)$ satisfies $\ell(u) \leq c_\ell$, $\rho$-almost everywhere. 

Using the macroscopic velocity and density, we can represent the alignment force as follows
\[
\cA (\mu)(x,v) =   \bc_\rho(u)(x) - v \st_\rho(x).
\]

\end{remark}

One example of a measure-valued weak solution is provided by the empirical measure
 \begin{equation}\label{e:empirical}
\mu^N = \sum_{i=1}^N m_i \d_{x_i(t)} \otimes \d_{v_i(t)},
\end{equation}
where $\{(x_i,v_i)\}_i$ solve the agent-based system 
\begin{equation}\label{e:ABS}
\left\{\begin{split}\dot{x}_i  &= v_i, \\
\dot{v}_i  &= \sum_{j=1}^N m_j \phi_{\rho^N} (x_i,x_j) (v_j - v_i), \quad \rho^N = \sum_{i=1}^N m_i \d_{x_i},
\end{split}\right.
\end{equation}  
provided the protocol $\cC$ admits Dirac densities $\cP_q \ss \cF$. 

If \eqref{e:ABS} is Hadamard well-posed, soltions clearly satisfy the classical maximum principle:
\begin{equation*}\label{}
\begin{split}
\max_i |v_i(t)|  & \leq \max_i |v_i(0)|, \quad t>0,\\
\max_i |x_i(t) - x_i(0)|  & \leq t \max_i |v_i(0)|.
\end{split}
\end{equation*}

It is natural to expect that a similar maximum principle  holds for any weak solution to \eqref{e:VAmu}.  This is not so straightforward, however, knowing that for  rough coefficients the equation has  no well-defined characteristic map  to represent the solution as a push-forward of the initial data. 

Let us discuss the classical picture vs general setting suitable for bounded protocols.

\subsection{Classical transport structure}
Let us suppose for a moment that  for a given weak solution $\mu$ the coefficients are bounded and Lipschitz
\begin{equation}\label{ }
\st_\rho, \bc_\rho(u)  \in L^\infty([0,T); \Lip_\loc(\O)  \cap L^\infty(\O)).
\end{equation}
This would generally be the case, for example, if the kernel $\phi_\rho$ is Lipschitz in $x$ uniformly in $y$.  In this case we have a globally well-posed  characteristic system  for a pair $X = X(t,s,x,v)$, $V = V(t,s,x,v) \in \Lip_\loc(\domain)$ solving
\begin{equation}\label{e:char}
\begin{split}
\p_t X = V, & \quad \p_t V = \bc_\rho(u)(X) - \st_\rho(X) V,\\
X(s,s,x,v)  = x \in \O,  & \quad V(s,s, x,v)  = v \in \R^n.
\end{split}
\end{equation}
As a consequence, we have the classical transport representation of the solution.
\begin{lemma}\label{l:transport}
Suppose $\st_\rho, \bc_\rho(u)  \in L^\infty([0,T); \Lip_\loc \cap L^\infty(\O)))$. Then 
\begin{equation}\label{e:push}
\mu(t) = (X,V) \# \mu(s), \quad 0\leq s<t<T,
\end{equation}
i.e. for any $h\in C_b(\domain)$,
\begin{equation}\label{e:consK}
\int_{\domain} h(x,v) \dmu(t,x,v) = \int_{\domain} h(X(t,s,x,v),V(t,s,x,v)) \dmu(s,x,v).
\end{equation}
\end{lemma}
\begin{proof}
We can see that by solving the backward transport equation
\begin{equation}\label{e:backward}
\p_s g + v\cdot \n_x  g + \cA(\mu) \cdot \n_v g = 0,
\end{equation}
with terminal condition $g(t,x,v) = h(x,v)$ for an arbitrary $h\in C_0^1(\domain)$. Indeed, the solution is given by $g(s,x,v) = h(X(t,s,x,v),V(t,s,x,v))$. 
So, plugging it into \eqref{e:weakVA} we obtain \eqref{e:consK}. Extension to $h\in C_b(\domain)$ is done by the standard approximation.
\end{proof}

Using  \eqref{e:push} we can formally rewrite the characteristic system \eqref{e:char} (here $s=0$) in a form reminiscent of the discrete system \eqref{e:ABS}:
\begin{equation}\label{e:XV}
\p_t X  = V,  \quad  \p_t V  =  \int_\domain \phi_\rho(X,X') (V'-V) \dmu_0',
\end{equation}
where $X= X(t,0,x,v)$, $X' = X(t,0,x',v')$, etc. A transition from \eqref{e:char} to \eqref{e:XV}, however, requires $h(\cdot) = \phi_\rho(X, \cdot)$ to be bounded and continuous.  These assumptions are not always satisfied, especially the boundedness even if the defining kernel is smooth, c.f. the \ref{MT}-protocol on $\R^n$. If $\supp \mu$ is uniformly bounded on $[0,T) \times \domain$, then local continuity of $\phi_\rho$ implies \eqref{e:XV} by the standard truncation argument.
\begin{lemma}\label{l:chartrad}
Let $\st_\rho, \bc_\rho(u)  \in L^\infty([0,T); \Lip_\loc \cap L^\infty(\O)))$, $\phi_\rho (x,\cdot) \in C_\loc(\O)$ for all $x\in \O$, and $\supp \mu \ss [0,T) \times B_R$. Then the system \eqref{e:char} is equivalent to \eqref{e:XV}.
\end{lemma}

Under the assumptions of \lem{l:chartrad}  the characteristic system takes traditional form \eqref{e:XV}, and the maximum principle on $V$-component follows directly. Under the lower bound on the kernel \eqref{e:kerbelowintro} the flocking follows too by the Ha-Liu argument \cite{HL2009}. So, in the smooth settings the theory falls under the scope of existing literature.

\subsection{Probabilistic representation of solutions}\label{s:transport}  
Let $\mu \in L^\infty([0,T); \cP_q(\O \times \R^n))$ be a weak solution to \eqref{e:VA} without any additional regularity assumptions on the macroscopic parameters. The lack of regularity of the transport field $U$ makes the characteristic system \eqref{e:char}
ill-posed in the Hadamard sense. In particular, the characteristics may loose uniqueness. 

In this case, the push-forward representation of solutions given by \eqref{e:push} can still be interpreted in probabilistic sense where a single characteristic curve is replaced with a bundle of curves emanating from the same data and with an associated probability measure on the bundle. More specifically,  since our solution is narrowly continuous, for $q>1$, it falls under the scope of \cite[Theorem 8.2.1]{Ambrosio}, which states that there exists a measure $\Sigma$ on $C([0,T]; \domain) \times \domain$ supported on  the set $(X,V,x,v)$ where $(X,V)$ is an absolutely continuous distributional solution to \eqref{e:char} (with $s=0$), such that the following representation holds:
\begin{equation}\label{e:pushmu}
\int_\domain h(x,v) \dmu(t,x,v) = \int_{C([0,T]; \domain) \times \domain} h(X(t),V(t)) \dSigma(X,V,x,v), \quad \forall h\in C_b(\domain).
\end{equation}

The representation implies several structural properties of the measure $\Sigma$. First, 
evaluating \eqref{e:pushmu} at $t=0$ shows that the $(x,v)$-marginal of $\Sigma$ is $\mu_0$:
\[
\int_{C([0,T]; \domain)} \dSigma(X,V,x,v) = \mu_0(x,v).
\]
This in particular implies that $\Sigma$ is supported on initial conditions starting from the support of $\mu_0$. Second, by the disintegration theorem, there exists a family of probability measures $\{ \nu_{x,v}\}_{(x,v) \in \supp \mu_0}$ on the path space $C([0,T]; \domain)$ each supported on absolutely continuous characteristic curves \eqref{e:char} starting with the same initial data $(x,v)$, such that for all $h\in C_b(\domain)$,
\begin{equation}\label{e:disint}
\int_\domain h(x,v) \dmu(t,x,v) = \int_{ \domain} \left( \int_{C([0,T]; \domain)} h(X(t),V(t)) \dnu_{x,v}(X,V) \right) \dmu_0(x,v).
\end{equation}
In the classical situation, $\nu_{x,v}$ would be the Dirac measure on the unique characteristics $(X,V)$ starting from $(x,v)$. 

Third, we obtain a somewhat instructive description of the support of $\mu$ at time $t$:
\begin{equation}\label{e:suppmu}
\supp \mu(t) \ss \overline{ \{ (X(t),V(t)):  (X(0),V(0)) \in \supp \mu_0 \text{ and } (X,V) \text{ solves } \eqref{e:char} \}}.
\end{equation}
Indeed, if $h$ is supported outside of the set on the right hand side, then clearly from \eqref{e:disint}, we obtain 
$\int_\domain h(x,v) \dmu(t,x,v) = 0$. 

Further desired implications of \eqref{e:disint}, such as the maximum principle or the flocking analysis, which traditionally relies on the push-forward representation of solutions stumbles upon the lack of regularity of the field $U$ and continuity of the characteristic map.  For example, in writing the equation for $V$-characteristics we use the transport nature of the solution to represent it in terms of the initial measure:
\[
\p_t V = \int_\domain \phi_\rho(X,x') (v' - V)\dmu(t,x',v') = \int_{C([0,T]; \domain) \times \domain}  \phi_\rho(X,X'(t)) (V'(t) - V)\dSigma(X',V',x',v').
\]
This already requires the function $h(x',v') = \phi_\rho(X,x') (v' -V)$ to be continuous in $x',v'$, which is not always the case.  Furthermore, due to lack of continuity with respect to initial data, it is not always possible to find a representative $V$ which maximizes a particular coordinate of $V$ -- a necessary step to prove the maximum principle via the traditional technique. Similarly, the diameter of the $X$-characteristics $D = \sup\{ |X(t,x,v) - X(t,x',v')|: (x,v), (x',v') \in \supp \mu_0 \}$ is not always achieved.

For these reasons we must depart from the traditional characteristic-based analysis and develop a new technique which relies solely on the original weak formulation of the solution \eqref{e:weakVA}.

We finish our introductory discussion with a similar remark pertaining to representation of solutions to the macroscopic density equation
\begin{equation}\label{ }
\p_t \rho + \n_x(u \rho) = 0
\end{equation}
which is a direct consequence of \eqref{e:weakVA} by taking test functions $g = g(t,x)$. The finite $q$-th moment implies that the family of measures $\rho(t)$, $t<T$ is tight, and weakly continuous, even Lipschitz in $W_q$. Consequently, it is narrowly continuous. So, if $q>1$, and hence $u\in L^\infty L^q(\rho)$ according to \eqref{e:umuLp}, we apply \cite[Theorem 8.2.1]{Ambrosio},  which states that there exists a probability measure $\eta$ on $C([0,T]; \O) \times \O$ satisfying  supported on the set of pairs $(\g,x)$ where $\g$ is an absolutely continuous distributional solution to 
\begin{equation}\label{e:gODE}
\ddt \g(t) = u(\g(t),t), \quad \text{with } \g(0) = x;
\end{equation}
and the following representation holds
\begin{equation}\label{e:etarep}
\int_\O h(x) \drho(t,x) = \int_{C([0,T]; \O)  \times \O} h(\g(t)) \deta(\g,x), \quad \forall h\in C_b(\O), \ t <T.
\end{equation}

Similarly to the previous discussion, we have 
\begin{equation}\label{ }
\rho_0 = \int_{C([0,T]; \O)} \deta(\g,\cdot),
\end{equation} 
and 
\begin{equation}\label{e:supprho}
\supp \rho(t) \ss  \overline{ \{ \gamma(t):  \g(0) \in \supp \rho_0 \text{ and } \g \text{ solves } \eqref{e:gODE} \}}.
\end{equation}

\subsection{Maximum principle for weak solutions} \label{ss:mp}

Let us state a version of  the maximum principle for weak measure-valued solutions.

\begin{proposition}[Weak Maximum Principle]\label{p:mp} Suppose $\cC$ is a bounded communication protocol. 
Let $\mu \in L^\infty([0,T); \cP_1(\O \times \R^n))$ be an admissible weak solution to \eqref{e:VA} with a compact initial support $\supp \mu_0 \ss \cO \times \cK$, where $\cK \ss \R^n$ is assumed to be convex and bounded.  Then 
\begin{equation}\label{e:mp}
\supp \mu(t) \ss (\cO +t \cK) \times \cK.
\end{equation}
In particular, the support of the solution in $v$-direction will remain confined to its initial convex hull.
\end{proposition}

\begin{proof} 
The proof will be done in two steps leading towards \eqref{e:mp}.

{\sc Step 1}: $\supp \mu(t) \ss \O \times \cK$.   So, here we will focus solely on establishing that the $v$-support of the solution, i.e. $\supp_v \mu = \supp (\int_\O \dmu(x,\cdot))$, is confined to $\cK$.
To this end, we recall that
any convex compact set $\cK$ can be represented as an intersection of hyperspaces:
\[
\cK = \bigcap_{\ell \in F \ss \R^n} \{v:  \ell(v) \leq c_\ell\},
\]
for some compact $F \ss \R^n$. So, we aim to establish that if initially for a fixed $\ell\in F$, $\ell(\supp_v \mu_0 ) \leq c_\ell$, then the same holds at any later times. Without loss of generality and to simplify the notation, let us just assume that $\ell$ is the $i$th basis functional $\ell(v) = v_i$, and denote $c_\ell = c_i$. So, we must demonstrate  $(\supp_v \mu(t) )_i \leq c_i$ if initially so.

To this end, let us consider the test-function $g= g(v_i - c_i)$, where $g \in C^1_\loc$ is given by
\begin{equation*}\label{}
g(r) = \left\{ \begin{split}
0, & \quad r<0,\\
 \frac12 r^2, &\quad 0\leq r \leq 1,\\
r-\frac12, &\quad r >1.
\end{split}\right.
\end{equation*}
Basically, we want a function supported on $r>0$ with positive first and second derivatives $g',g'' \geq 0$, and so that $r g'(r)\leq C g(r)$.  Since $g$ fulfills the assumptions of \rem{r:ext1} we obtain from the weak formulation

\begin{equation*}\label{}
\begin{split}
\int_{\domain} g(v_i - c_i) \dmu(t,x,v) & =	 \int_0^t \int_{\domain} \int_{\domain} \phi_\rho(x,y) (w_i  - v_i) g'(v_i - c_i)\dmu(s,y,w) \dmu(s,x,v) \ds \\
& \leq 	 \int_0^t \int_{\domain} \int_{\domain} \phi_\rho(x,y) (w_i  - v_i) g'(v_i - c_i) \one_{w_i > v_i}\dmu(s,y,w) \dmu(s,x,v) \ds \\
& \leq  \int_0^t \int_{\domain} \int_{\domain} \phi_\rho(x,y) (w_i - c_i) g'(v_i - c_i) \one_{w_i > v_i} \dmu(s,y,w)\dmu(s,x,v) \ds \\
& \leq  \int_0^t \int_{\domain} \int_{\domain} \phi_\rho(x,y) (w_i - c_i) g'(w_i - c_i) \one_{w_i > v_i} \dmu(s,y,w)\dmu(s,x,v) \ds \\
& \leq  \int_0^t \int_{\domain} \int_{\domain} \phi_\rho(x,y) g(w_i - c_i) \dmu(s,y,w)\dmu(s,x,v) \ds \\
& \leq  \int_0^t \int_{\domain}  \st^*_\rho(y) g(w_i - c_i) \dmu(s,y,w) \ds\\
& \leq  \oS^* \int_0^t \int_{\domain}   g(w_i - c_i) \dmu(s,y,w) \ds.
\end{split}
\end{equation*}
By \GL\ $\int_{\domain} g(v_i - c_i) \dmu(t,x,v) = 0$, as desired, which implies that $\mu$ is supported on the hyperspace $v_i \leq c_i$, as desired.

{\sc Step 2}: $\supp \rho \ss \cO + t \cK$. As a consequence of \rem{r:macro},  $u\in \cK$, $\rho$-almost everywhere. Note that in the same density $\rho$ satisfies the continuity equation weakly $\p_t \rho + \n_x( \tilde{u} \rho)$ for any $\tilde{u} = u$, $\rho$-almost everywhere. So, let us alter $u$ on a $\rho$-negligible set such that  $u(x)\in \cK$, for any $x$ (we can set $u = \bar{u}$ outside the $\supp \rho$). 

Now, integrating \eqref{e:gODE} in time,
\begin{equation}\label{e:gammaFTC}
\g(t) = x + \int_0^t u(\g(s),s)\ds \in \cO + t \cK.
\end{equation}
Hence, $\supp \rho \ss \cO + t \cK$ as follows from \eqref{e:supprho}.
\end{proof}

\subsection{Cucker-Smale flocking theorem for weak solutions}\label{s:CSFT}

\begin{proposition}\label{p:CS} Suppose the kernel $\phi_\rho$ has the following power bound from below
\begin{equation}\label{e:kerbelow}
\phi_\rho(x,y) \geq \frac{\l}{\jap{x-y}^\b}, \quad \rho \otimes \rho - \text{a.e.}
\end{equation}
for some $\b \leq 1$ and $\l>0$. 
Let $\mu \in L^\infty([0,T); \cP_1(\O \times \R^n))$ be any weak solution to \eqref{e:VA} with a compact initial support $\supp \mu_0 \ss \cO \times \cK$. Then the diameter of the flock remains bounded
\[
D(t): =  \sup\{ |x-y|: \ x,y \in \supp \rho(t)\} <\infty,
\]
and exponential alignment holds
\[
A(t): = \sup\{ |v-w|: \  v,w \in \supp_v \mu(t) \} \leq A(0) e^{-\d t},
\]
for some $\d >0$.
\end{proposition}

Before we get to the proof, let us expand the class of test-functions that we are allowed to use in the weak formulation \eqref{e:weakVA}.

\begin{lemma}\label{l:ext2}
The definition \defin{d:weakVA} extends to include test-functions $g \in W^{1,\infty}([0,T)\times \O \times \R^{n})$ such that  $\p_x g, \p_v g \in C([0,T)\times \O \times \R^{n})$ and  $\p_t g \in L^\infty_t C_{x,v}$. 
\end{lemma}
\begin{proof}
We prove the lemma by showing that \eqref{e:weakVA} holds for any small initial time $s<t$ and then letting $s \to 0$, noting that by weak continuity of the measures and strong continuity of the test-function
\[
\int_{\domain} g(s,x,v)\dmu(s,x,v) \to \int_{\domain} g(0,x,v)\dmu_0(x,v), \quad s \to 0.
\]

Consider a standard 1D mollifier $\chi$ supported on the unit interval, and for $0<\d<s$ let $\chi_\d = \d^{-1} \chi(t \d^{-1})$, $g_\d(t) = g \ast \chi_\d(t)$, for $t\geq s$. Then $g_\d \in C^1$ and we can plug it into \eqref{e:weakVA}:
	\begin{equation}\label{e:weakVAd}
	\begin{split}
\int_{\domain} g_\d(t,x,v)\dmu(t,x,v)  &=	\int_{\domain} g_\d(s,x,v)\dmu(s,x,v) \\
&+ \int_s^t \int_{\domain}  ( \p_\t g_\d + v\cdot \n_x  g_\d + \cA(\mu) \cdot \n_v g_\d)\dmu(\t,x,v)\dtau .
\end{split}
\end{equation}
The end-point integrals as well as the inner integral involving $v\cdot \n_x  g_\d + \cA(\mu) \cdot \n_v g_\d$ can be interpreted as a pairing of the form $\int_M G_\d(t) d \lambda(t)$, where $\lambda$ is a finite measure and $G$ is a continuous function on $M$. Since $G_\d \to G$ pointwise and boundedly,  the dominated convergence applies to get 
\[
 \int_s^t \int_{\domain}  ( v\cdot \n_x  g_\d + \cA(\mu) \cdot \n_v g_\d)\dmu(\t,x,v)\dtau \to  \int_s^t \int_{\domain}  (  v\cdot \n_x  g + \cA(\mu) \cdot \n_v g)\dmu(\t,x,v)\dtau.
\]
As to the remaining integral, we have
\begin{equation*}\label{}
\begin{split}
\int_s^t \int_{\domain}  \p_\t g_\d \dmu(\t,x,v) \dtau & =\int_{-s}^s \int_s^t \int_{\domain}  \p_\t g(\t - \t',x,v) \dmu(\t,x,v) \dtau \chi_\d(\t') \dtau' \\
& =\int_{-s}^s \int_{s+\t'}^{t+\t'} \int_{\domain}  \p_\t g(\t ,x,v) \dmu(\t+\t',x,v) \dtau \chi_\d(\t') \dtau' .
\end{split}
\end{equation*}
Since $\p_\t g(\t ,x,v) \in C(\domain)$ and $t \to \mu(t)$ is weakly on continuous, we have 
\[
\int_{\domain}  \p_\t g(\t ,x,v) \dmu(\t+\t',x,v)  \to \int_{\domain}  \p_\t g(\t ,x,v) \dmu(\t,x,v),
\]
for any $\tau>0$, while clearly the integrals are bounded uniformly in $\tau$. So, by the dominated convergence,
\[
\int_{s+\t'}^{t+\t'} \int_{\domain}  \p_\t g(\t ,x,v) \dmu(\t+\t',x,v) \dtau \to \int_{s}^{t} \int_{\domain}  \p_\t g(\t ,x,v) \dmu(\t,x,v) \dtau.
\]
Consequently,
\[
\int_s^t \int_{\domain}  \p_\t g_\d \dmu(\t,x,v) \dtau \to \int_s^t \int_{\domain}  \p_\t g \dmu(\t,x,v) \dtau.
\]
\end{proof}
\begin{lemma}\label{l:ext3}
Let $e \in C^1(\R)$ be a smooth function of time and $\f \in C^1(\R^n)$. Then the function
\[
g(t,v) = \frac12 \int_\domain \f(v_i - v_i' - e(t)) \dmu(t,x',v')
\]
satisfies the conditions of \lem{l:ext2}.
\end{lemma}
\begin{proof} All the conditions are trivial except that $\p_t g \in L^\infty_t C_{v}$. Let us check it.

Clearly, for a fixed $v$,  the function $(t,v') \to \f(v_i - v_i' - e(t))$ is $C^1$, and therefore can be used as a test-function in \eqref{e:weakVA}. We obtain
\begin{equation*}\label{}
\begin{split}
g(t,v) - g(s,v) & =  - \int_s^t \int_{\domain}   \dot{e}(\t) \f'(v_i - v_i' - e(\tau)) \dmu(\t,x',v')\dtau \\
&-\int_s^t \iint_{\ddomain}   \phi_\rho(x',x'')(v_i'' - v_i') \f'(v_i - v_i' - e(\tau)) \dmu'' \dmu' \dtau.
\end{split}
\end{equation*}
The inner spacial integrals are uniformly bounded as functions of $(\t,v)$. It shows that $g$ is Lipschitz in $t$ and that
\begin{equation}\label{e:timepg}
\p_t g(t,v) = -\int_{\domain}   \dot{e}(t) \f'(v_i - v_i' - e(t)) \dmu'- \iint_{\ddomain}   \phi_\rho(x',x'')(v_i'' - v_i') \f'(v_i - v_i' - e(t)) \dmu'' \dmu' ,
\end{equation}
which clearly belongs to $L^\infty_t C_{v}$.
\end{proof}

\begin{proof}[Proof of \prop{p:CS}] According to the maximum principle proved in \prop{p:mp}, $D(t) \leq a + bt:= d(t) $ for some $a,b>0$.
Denote $\phi(r) =  \frac{\l}{\jap{r}^\b}$, and define
\[
e(t) = A(0) \exp\left\{ - \e \int_{d(0)}^{d(t)} \phi(r) \dr \right\},
\]
where $\e>0$ is small to be defined later.  We first will demonstrate that $|v-w| \leq e(t)$, for all $(v,w)\in \supp_v \mu \otimes \mu$. This will be done by showing that 
\[
G(t) = \frac12 \iint_{\ddomain} (v_i - v_i' - e(t))_+^2 \dmu' \dmu = 0,
\]
Since we already know that the support of $\mu$ in $v$ is bounded uniformly in time, we can redefine and taper the function $(v_i - v_i' - e(t))_+^2 $ for large $v,v'$ without changing $G(t)$.  We will not change the notation for such tapered function to simplify the notation.  Then, $G(t) = \int_\domain g(t,v) \dmu$, where 
\[
g(t,v) = \frac12 \int_\domain (v_i - v_i' - e(t))_+^2 \dmu',
\]
 and according to \lem{l:ext3}, $g$ can be used as a test-function. Since $G(0) = 0$, we obtain
\begin{equation*}\label{}
\begin{split}
G(t) = \int_0^t \int_{\domain}  ( \p_s g + \cA(\mu) \cdot \n_v g)\dmu \ds.
\end{split}
\end{equation*}
Using the formula for the time derivative of $g$, \eqref{e:timepg}, and that $\dot{e} = -\e b A_0 \phi(d) e$,
\begin{equation}\label{e:G1st}
\begin{split}
G(t) & = \e b A_0 \int_0^t \phi(d(s)) e(s)  \iint_{\ddomain}  (v_i - v_i' - e)_+ \dmu'\dmu\ds \\
& + \int_0^t \iiint_{\dddomain}   \phi_\rho(x',x'')(v_i' - v_i'') (v_i - v_i' - e)_+ \dmu'' \dmu'  \dmu \ds\\
& +\int_0^t \iiint_{\dddomain}   \phi_\rho(x,x'')(v_i'' - v_i) (v_i - v_i' - e)_+ \dmu'' \dmu'  \dmu \ds.
\end{split}
\end{equation}
We will fix an $\e>0$ so that $\e b A_0 \leq 1$. Thus, we have
\begin{equation}\label{e:Gsplit}
\begin{split}
G(t) & \leq R + I + II,\\
R & = \int_0^t \phi(d(s)) e(s)  \iint_{\ddomain}  (v_i - v_i' - e)_+ \dmu'\dmu\ds\\
I & = \int_0^t \iiint_{\dddomain}   \phi_\rho(x',x'')(v_i' - v_i'') (v_i - v_i' - e)_+ \dmu'' \dmu'  \dmu \ds,\\
II & = \int_0^t \iiint_{\dddomain}   \phi_\rho(x,x'')(v_i'' - v_i) (v_i - v_i' - e)_+ \dmu'' \dmu'  \dmu \ds.
\end{split}
\end{equation}
We will now split the integrals $I, II$ further. Let us start with $II$:
\begin{equation*}\label{}
\begin{split}
II & =  \int_0^t \iiint_{\dddomain}   \phi_\rho(x,x'')(v_i'' - v_i) (v_i - v_i' - e)_+ \one_{v_i'' -v_i' >e} \dmu'' \dmu'  \dmu \ds\\
& +   \int_0^t \iiint_{\dddomain}   \phi_\rho(x,x'')(v_i'' - v_i) (v_i - v_i' - e)_+ \one_{v_i'' -v_i' \leq e} \dmu'' \dmu'  \dmu \ds: = II_1 + II_2.
\end{split}
\end{equation*}
For the first part we have
\begin{equation}\label{e:II1}
\begin{split}
II_1 & =  \int_0^t \iiint_{\dddomain}   \phi_\rho(x,x'')((v_i'' - v_i' - e) - (v_i-v_i' -e)) (v_i - v_i' - e)_+ \one_{v_i'' -v_i' >e} \dmu'' \dmu'  \dmu \ds\\
& =  \int_0^t \iiint_{\dddomain}   \phi_\rho(x,x'')((v_i'' - v_i' - e)_+ - (v_i-v_i' -e)_+) (v_i - v_i' - e)_+ \one_{v_i'' -v_i' >e} \dmu'' \dmu'  \dmu \ds\\
& \leq \frac12 \int_0^t \iiint_{\dddomain}   \phi_\rho(x,x'')(v_i'' - v_i' - e)^2_+  \dmu'' \dmu'  \dmu \ds \\
& = \frac12 \int_0^t \iint_{\ddomain}   \st_\rho(x'') (v_i'' - v_i' - e)^2_+  \dmu'' \dmu' \ds \leq C \int_0^t G(s) \ds.
\end{split}
\end{equation}
Inside $II_2$, we have a negative integrand: $v_i'' - v_i \leq 0$. So, by our assumption on the kernel 
\eqref{e:kerbelow} we obtain
\[
II_2 \leq  \int_0^t \phi(D(s)) \iiint_{\dddomain}   (v_i'' - v_i) (v_i - v_i' - e)_+ \one_{v_i'' -v_i' \leq e} \dmu'' \dmu'  \dmu \ds.
\]
Writing $ \one_{v_i'' -v_i' \leq e} = 1 -  \one_{v_i'' -v_i' > e}$ we split the integral into two:
\begin{equation*}\label{}
\begin{split}
II_2 & \leq \int_0^t \phi(D(s)) \iiint_{\dddomain}   (v_i'' - v_i) (v_i - v_i' - e)_+ \dmu'' \dmu'  \dmu \ds \\
& + \int_0^t \phi(D(s)) \iiint_{\dddomain}   (v_i - v''_i) (v_i - v_i' - e)_+ \one_{v_i'' -v_i' > e} \dmu'' \dmu'  \dmu \ds: = II_{21} + II_{22}.
\end{split}
\end{equation*}
For the second, we perform similar estimates as in \eqref{e:II1}:
\begin{equation}\label{e:II22}
\begin{split}
II_{22} & = \int_0^t \phi(D(s)) \iiint_{\dddomain}   ((v_i - v_i' - e)- (v''_i-v'_i-e)) (v_i - v_i' - e)_+ \one_{v_i'' -v_i' > e} \dmu'' \dmu'  \dmu \ds  \\
& =  \int_0^t \phi(D(s)) \iiint_{\dddomain}   ((v_i - v_i' - e)_+- (v''_i-v'_i-e)_+) (v_i - v_i' - e)_+ \one_{v_i'' -v_i' > e} \dmu'' \dmu'  \dmu \ds  \\
& \leq C \int_0^t \iiint_{\dddomain} (v_i - v_i' - e)^2_+  \dmu'' \dmu'  \dmu \ds + C\int_0^t \iiint_{\dddomain}  (v''_i-v'_i-e)^2_+  \dmu'' \dmu'  \dmu \ds\\
& = C \int_0^t \iint_{\ddomain} (v_i - v_i' - e)^2_+   \dmu'  \dmu \ds + C\int_0^t \iint_{\ddomain}  (v''_i-v'_i-e)^2_+  \dmu'' \dmu'   \ds \\
& = C \int_0^t G(s) \ds.
\end{split}
\end{equation}
As to $II_{21}$ it provides the necessary dissipation: 
\begin{equation*}\label{}
\begin{split}
II_{21} & = \int_0^t \phi(D(s)) \iiint_{\dddomain}   (v_i'' - v_i) (v_i - v_i' - e)_+ \dmu'' \dmu'  \dmu \ds \\
& = \int_0^t \phi(D(s)) \iiint_{\dddomain}   (v_i'' - v'_i) (v_i - v_i' - e)_+ \dmu'' \dmu'  \dmu \ds \\
& +\int_0^t \phi(D(s)) \iiint_{\dddomain}   (v_i' - v_i) (v_i - v_i' - e)_+ \dmu'' \dmu'  \dmu \ds: = II_{211} + II_{212}.
\end{split}
\end{equation*}
We will leave $II_{211}$ unchanged for now, while inside $II_{212}$  we notice that $v_i' - v_i \leq -e$. 
So,
\begin{equation}\label{e:II212}
II_{212} \leq - \int_0^t \phi(D(s)) e(s) \iint_{\ddomain}   (v_i - v_i' - e)_+ \dmu'  \dmu \ds \leq  - R,
\end{equation}
where we recall that $R$ is defined in \eqref{e:Gsplit}. So, this term provides essential cancellation of what otherwise would be an uncontrolled term.

Let us summarize what we have obtained so far:
\begin{equation}\label{ }
G(t) \leq I+ II_{211} + C \int_0^t G(s) \ds.
\end{equation}
We will now utilize $I$ to cancel $II_{211}$ as well by cost of adding another Gr\"onwall term $\int_0^t G(s) \ds$. Indeed,  we proceed similarly,
\begin{equation*}\label{}
\begin{split}
I & = \int_0^t \iiint_{\dddomain}   \phi_\rho(x',x'')(v_i' - v_i'') (v_i - v_i' - e)_+ \dmu'' \dmu'  \dmu \ds\\
& = \int_0^t \iiint_{\dddomain}   \phi_\rho(x',x'')(v_i' - v_i'') (v_i - v_i' - e)_+ \one_{v_i'' >v_i - e}\dmu'' \dmu'  \dmu \ds\\
& + \int_0^t \iiint_{\dddomain}   \phi_\rho(x',x'')(v_i' - v_i'') (v_i - v_i' - e)_+ \one_{v_i''  \leq v_i - e}\dmu'' \dmu'  \dmu \ds \\
\intertext{noticing that the integrand in the first integral is negative,}
& \leq \int_0^t  \phi(D(s)) \iiint_{\dddomain}  (v_i' - v_i'') (v_i - v_i' - e)_+ \one_{v_i'' >v_i - e}\dmu'' \dmu'  \dmu \ds\\
& + \int_0^t \iiint_{\dddomain}   \phi_\rho(x',x'')((v_i-v''_i-e) - (v_i-v'_i-e)) (v_i - v_i' - e)_+ \one_{v_i''  \leq v_i - e}\dmu'' \dmu'  \dmu \ds \\
& = - II_{211}  - \int_0^t  \phi(D(s)) \iiint_{\dddomain}  (v_i' - v_i'') (v_i - v_i' - e)_+ \one_{v_i'' \leq v_i - e}\dmu'' \dmu'  \dmu \ds\\
& + \int_0^t \iiint_{\dddomain}   \phi_\rho(x',x'')((v_i-v''_i-e)_+ - (v_i-v'_i-e)_+) (v_i - v_i' - e)_+ \one_{v_i''  \leq v_i - e}\dmu'' \dmu'  \dmu \ds 
\intertext{the integral in the middle is manipulated similar to the last integral,}
& = - II_{211}  - \int_0^t  \phi(D(s)) \iiint_{\dddomain}  ((v_i-v''_i-e)_+ - (v_i-v'_i-e)_+) (v_i - v_i' - e)_+ \one_{v_i'' \leq v_i - e}\dmu'' \dmu'  \dmu \ds\\
& + \int_0^t \iiint_{\dddomain}   \phi_\rho(x',x'')((v_i-v''_i-e)_+ - (v_i-v'_i-e)_+) (v_i - v_i' - e)_+ \one_{v_i''  \leq v_i - e}\dmu'' \dmu'  \dmu \ds 
\intertext{after which in both integrals we make the same estimates as in \eqref{e:II1} and \eqref{e:II22}:} 
& \leq - II_{211} + C\int_0^t G(s) \ds.
\end{split}
\end{equation*}
We have proved that 
\[
G(t) \leq C\int_0^t G(s) \ds,
\]
which implies that $G(t) \equiv 0$. It proves that the diameter of the support of $\mu$ in the $v_i$-direction is bounded by $e(t)$. Doing it for all $i$ proves that $A(t) \leq C e(t)$. 

Let us now consider the case when $\b <1$. Then for some $\d>0$, we obtain $e(t) \lesssim e^{- c t^\d}$, for some $c,\d>0$. By the argument in the beginning of \sect{ss:mp}, $\|u - \bar{u}\|_{L^\infty(\rho)} \lesssim e^{- c t^\d}$. Changing $u$ possibly on a subset of zero $\rho$-measure, and defining $u = \bar{u}$ outside of $\supp \rho$, we can conclude that $\sup_{x\in \O} |u(x,t) - \bar{u}| \lesssim e^{- c t^\d}$. This alteration does not change the weak formulation of the continuity equation $\rho_t + \n_x( u \rho) = 0$. Hence,  as a consequence of the application of \eqref{e:supprho} and \eqref{e:gammaFTC}, $D(t) \leq \bar{D}<\infty$ for all time. 

We now repeat the above argument, this time choosing $e(t) = A_0 e^{-\e t}$ for some small $\e>0$. 
The residual term in \eqref{e:G1st} becomes
\[
\e A_0 \int_0^t  e(s)  \iint_{\ddomain}  (v_i - v_i' - e)_+ \dmu'\dmu\ds,
\]
while the estimate for the dissipation term in \eqref{e:II212} becomes
\begin{equation}\label{e:II212a}
II_{212} \leq - \phi(\bar{D}) \int_0^t  e(s) \iint_{\ddomain}   (v_i - v_i' - e)_+ \dmu'  \dmu \ds.
\end{equation}
So, picking $\e < \phi(\bar{D}) / A_0$ eliminates the residual term. This proves that $A(t) \leq e(t)$ as required.

Let us consider the case $\b=1$.  Then at the initial run, $e(t) \lesssim 1 / \jap{t}^\d$ for some $\d>0$. Consequently, the estimate on the diameter of the flock improves to $D \lesssim \jap{t}^{1-\d}$. The coefficient in front of the dissipation term improves to $\phi(D) \gtrsim \frac{1}{\jap{t}^{1-\d}}$, which means that we can pick $e(t) = A_0 e^{-\e (\jap{t}^\d-1)}$ to absorb the residual term. This brings us up to the previous case, because the amplitude $A(t) \leq e(t)$ becomes integrable in time, and the diameter $D$ becomes bounded. 

This finishes the proof.

\end{proof}

\begin{remark}\label{}
The theorem also holds under general fat tail condition:
\begin{equation}\label{e:kerbelow}
\phi_\rho(x,y) \geq \phi(|x-y|), \quad \rho \otimes \rho - \text{a.e.},
\end{equation}
where 
\[
\int_0^\infty \phi(r) \dr = \infty.
\]
However, it offers no specific rate of alignment:
\[
A(t) = o(1).
\]
This follows from the first run of the argument since $\int_{d(0)}^{d(t)} \phi(r) \dr \to \infty$. Such a weak alignment rate doesn't allow us to conclude that the radius of the flock $D$ remains bounded.

\end{remark}

\section{Existence of admissible weak solutions}

The results of Propositions \ref{p:mp} and \ref{p:CS} are stated for arbitrary weak solutions with compact initial data. Let us discuss existence of such solutions in the context of models based on bounded communication protocols.  At this point we will need all the assumptions of \defin{d:protocol}.

On the discrete level, by the classical Peano Theorem one can construct a class of weak solutions given by empirical measures \eqref{e:empirical} $\mu^N$ provided the kernel $\phi_{\rho^N}$ is a continuous function of $x_i$'s, and the admissible class $\cF$ contains Diracs.   This would be the case, for example, for the Cucker-Smale model \eqref{CS} with $\phi \in C_b(\O)$. Same for the Motsch-Tadmor model \eqref{MT} since in the kernel
\[
\phi_{\rho^N}(x_i,x_j) = \frac{\phi(x_i - x_j)}{\sum_k m_k \phi(x_i - x_k)}
\]
the denominator is bounded from below by $m_i \phi(0)$.

For integrable densities a general result can be established for compactly supported data in $L^p$, $p>1$. We prove it next.

\begin{theorem}\label{t:weak} Let $\cC$ be a bounded communication protocol.
For any compactly supported data $f_0 \in L^p(B_R \times B_R)$ for some $p>1$, there exists a global weak solution $f \in L^\infty([0,T); L^p(B_{R(T)} \times B_R)$ to \eqref{e:VA}. 
\end{theorem}

We note that this theorem is new new even for Cucker-Smale model and for Motsch-Tadmor model on the open space $\O = \R^n$ since the work of Karper-Mellet-Trivisa \cite{KMT2013} require the defining kernel $\phi$ to be local and compactly supported:
\begin{equation}\label{e:KMTrR}
\l \one_{|x|\leq r} \leq \phi(x) \leq \L \one_{|x|\leq R}.
\end{equation}
This in particular prohibits the application of the Cucker-Smale flocking theorem. 
Since any $L^\infty$ power-like kernel \eqref{e:powerlike} defines a bounded communication protocol, our theorem provides a class of solutions for the Motsch-Tadmor model to which \prop{p:CS} applies.

Since the data is compactly supported, it will remain such in the course of the construction. Therefore,  the cases $\O = \R^n$ and $\O = \T^n$ are similar. Let us assume $\O = \R^n$ for definiteness.

\subsection{Approximation scheme}
Let us consider a standard mollifier $\chi \in C^\infty_0(B_1(0))$, and $\chi_\e = \e^{-n} \chi(\e^{-1}x)$. We consider the mollification of the initial data
\[
f_\e(0,x,v) =  (\chi_\e \otimes \chi_\e) \ast f_0.
\]
Since $\supp f_0 \ss B_R \times B_R$, we have $\supp f_0^\e \ss B_{R+1} \times B_{R+1}$.  Now, let us fix a reference density
\[
\orho(x) = \frac{1}{(2\pi)^n} e^{- \frac{|x|^2}{2}},
\]
and let us introduce the regularized kernel
\[
\phi^\e_\rho = (\chi_\e \otimes \chi_\e) \ast \phi_{(1-\e)  \chi_\e \ast \rho + \e \orho}.
\]
Since by \ref{i:th2} the thickness of $(1-\e)  \chi_\e \ast \rho + \e \orho$ is positive on any ball $B_R$, by \ref{i:b1} the kernel is bounded, and hence $\phi^\e_\rho\in C^\infty_\loc(\O \times \O)$.
Furthermore, the kernel satisfies all the same axioms \ref{i:f1} -- \ref{i:f3} uniformly in $\e$ for $\e<1/2$ for uniformly admissible densities $\|\rho\|_\cF <C$:
\begin{align}\label{}
\st^\e_\rho(x) &= \int_\O \phi^\e_\rho(x,y) \drho(y) = \chi_\e \ast \left( \int_\O \phi_{(1-\e)  \chi_\e \ast \rho +\e \orho} (\cdot, y)  \chi_\e \ast \rho(y)  \dy  \right) \leq 2 \oS, \label{e:Se}\\
(\st^\e_\rho)^*(y) &= \int_\O \phi^\e_\rho(x,y) \drho(x) = \chi_\e \ast \left( \int_\O \phi_{(1-\e)  \chi_\e \ast \rho +\e \orho} (x, \cdot)  \chi_\e \ast \rho(x)  \dx  \right) \leq 2 \oS^*, \label{e:Se*}\\
\int_{\O \times \O} &(\phi^\e_\rho(x,y) )^2\drho(y) \drho(x)  \leq \int_{\O \times \O}  (\chi_\e \otimes \chi_\e) \ast \phi^2_{(1-\e)  \chi_\e \ast \rho + \e \orho} \drho(y) \drho(x) \label{e:He}\\
& = \int_{\O \times \O}  \phi^2_{(1-\e)  \chi_\e \ast \rho + \e \orho} (x,y) \chi_\e \ast \rho(y) \chi_\e \ast \rho(x)\dy \dx \leq 4\oH. \notag
\end{align}

Now let us consider the regularized equation with mollified data
\begin{equation}\label{e:VAe}
\begin{split}
\p_t f_\e + v \cdot \n_x f_\e +  \n_v \cdot (  \cA_\e(f_\e) f_\e) & = 0, \\
f^\e(0,x,v)  = f^\e_0(x,v) &.
\end{split}
\end{equation}
where
\begin{equation}\label{e:Ae}
\cA_\e(f_\e)(x,v) = \int_\domain \phi^\e_{\rho_\e}(x,x') (v'-v) f_\e(x',v') \dv'\dx'.
\end{equation}
There exists a global classical solution to \eqref{e:VAe} via the mean-field limit, see for example \cite{Shv-EA}, which satisfy the classical push-forward transport representation \eqref{e:push} along the smooth characteristic flow. The flow satisfies the classical maximum principle, and therefore, $f_\e \in C_0^\infty(B_{R t + R} \times B_R)$ for all $t>0$. This implies in particular that $\rho_\e \in L^1(B_{R T + R})$ uniformly in $\e$, and hence according to our assumptions of $\cF$,  $\| \rho_\e\|_\cF \leq C_T$. Consequently, in the protocol bounds \ref{i:f1} - \ref{i:f3}, \ref{i:b1} - \ref{i:b2} dependencies on $\| \rho_\e\|_\cF$ can be eliminated. 

Next, testing the equation with $f^{p-1}_\e$ we obtain after integrating by parts in the alignment term
\begin{equation*}\label{}
\begin{split}
\ddt \int_\domain f^p_\e \dv\dx  & = n  \int_\domain  \int_\domain \phi^\e_{\rho}(x,x') f_\e^p f_\e' \dv'\dx' \dv\dx \\
& =  n  \int_\domain  \int_\O \phi^\e_{\rho}(x,x')\rho(x') \dx' f_\e^p \dv\dx \leq 2n\oS \int_\domain f^p_\e \dv\dx.
\end{split}
\end{equation*}
This shows that $f_\e \in L^p(\domain)$ uniformly.  Now, recall that the supports of $f_\e$ are contained in a common $B_R \times B_R$.  It implies that $|u_\e| \leq R$ pointwise. Hence, $|\bc^\e_{\rho_\e}(u_\e)|^p \leq \oS^p R^p$ pointwise also.  This shows that
\[
|\cA_\e(f_\e)| \leq C_1 |v| + C_2,
\]
pointwise. Consequently, $\cA_\e(f_\e) f_\e \in L^p$ uniformly.

We are in a position to apply the averaging lemma from \cite{PS1998}.

 \begin{lemma} \label{l:al}
Suppose $f_n$ and $G_n$ are bounded in $L^p_\loc(\domain \times \R)$ solve the following equation on $\O \times \R^n \times \R$
\begin{equation}\label{ }
\p_t f_n + v \cdot \n_x f_n = \n_v G_n,
\end{equation}
Let $\f \in C^\infty_0(\R^n)$. Then the macroscopic quantities
\[
\rho^{\f}_n (t,x) = \int_{\R^{n}} \f(v) f_n(t,x,v) \dv
\]
are relatively compact in $L^p_\loc(\O \times \R)$. 
\end{lemma}

By the standard diagonalization argument and extending $f_\e$ to a solution on the real line using temporal cutoffs (see \cite{Shv-EA} as an example), we can select a sequence $f_{\e_n} = f_n$ such that $f_n \to f\in L^\infty([0,T);  L^p(B_R \times B_R))$ weakly, and
\begin{equation}\label{e:macroLp}
\| \rho^{\f}_n - \rho^{\f} \|_{L^p([0,T)\times \O)} \to 0,
\end{equation}
for all $\f \in C^\infty_0(\R^n)$. In particular, since all $f_n$ are supported on $B_R \times B_R$ by choosing $\f = 1$ on $B_R$ we obtain $\rho^{\f}_n = \rho_n$. So,
\[
\| \rho_n - \rho \|_{L^p([0,T)\times \O)} \to 0.
\]
Similarly, choosing $\f = v$ on $B_R$ we obtain $\rho^{\f}_n = \rho_n u_n$. So,
\[
\| \rho_n u_n - \rho u \|_{L^p([0,T)\times \O)} \to 0.
\]

\subsection{Obtaining a weak solution}
It remains to demonstrate that $f$ is a weak solution on $[0,T) \times \O \times \R^n$. By approximation we can assume that a test function  is given by $g= \chi(t,x) \f(v) \in C_0^\infty([0,T) \times \O \times \R^n)$. Plugging it into \eqref{e:weakVA} we have
\begin{equation}\label{e:weakn}
\begin{split}
 \int_0^T\int_{\O}  \rho_{\f}^n(t,x) \p_t \chi(t, x) \dx \dt  + & \int_{\O} \rho^n_\f(0,x) \chi(0,x) \dx+ \int_0^T\int_{\O}  \rho_{v \f}^n(t,x) \cdot \n_x \chi(t,x) \dx \dt \\
 = &  \int_0^T\int_{\domain} \cA_n(f_n) f_n  \chi(t,x) \n_v \f(v) \dv \dx \dt .
\end{split}
\end{equation}
Convergence of the left hand side is trivial by \eqref{e:macroLp}. Let us focus on the right hand side.

From the definition of $\cA_n$, \eqref{e:Ae}, and by fixing another test-function $\f' = v'$ on $B_R$, we have
\begin{equation*}\label{}
\begin{split}
 \int_0^T\int_{\domain} \cA_n(f_n) f_n  \chi(t,x) \n_v \f(v) \dv \dx &=  \int_0^T\int_{\O}  \int_\O \phi^n_{\rho_n}(x,x') \rho_n^{\f'}(x') \cdot \rho_n^{\n \f}(x) \chi(t,x) \dx' \dx \dt\\
 &   -  \int_0^T\int_{\O}  \int_\O \phi^n_{\rho_n}(x,x') \rho_n(x')   \rho_n^{v\cdot \n \f}(x) \chi(t,x) \dx' \dx \dt.
\end{split}
\end{equation*}

The two integrals on the right hand side are treated similarly. So, let us show that 
\[
\int_0^T\int_{\O}  \int_\O \phi^n_{\rho_n}(x,x') \rho_n^{\f'}(x') \cdot \rho_n^{\n \f}(x) \chi(t,x) \dx' \dx \dt \to \int_0^T\int_{\O}  \int_\O \phi_{\rho}(x,x') \rho^{\f'}(x') \cdot \rho^{\n \f}(x)  \chi(t,x) \dx' \dx \dt
\]

Let us observe that for any compactly test-function $\f$,  we have pointwise estimate
\[
| \rho^\f(x,t)| \leq C_\f \rho(x,t).
\]
Let us move the $\chi$-convolutions back on the densities:
\begin{multline*}
\int_0^T\int_{\O\times \O} \phi^n_{\rho_n}(x,x') \rho_n^{\f'}(x') \cdot \rho_n^{\n \f}(x)  \chi(t,x)  \dx' \dx \dt \\
= \int_0^T\int_{\O\times \O} \phi_{(1-\e_n)  \chi_{\e_n} \ast \rho_n + \e_n \orho}(x,x') \chi_{\e_n} \ast \rho_n^{\f'}(x') \cdot \chi_{\e_n} \ast ( \rho_n^{\n \f} \chi)(x) \dx' \dx \dt.
\end{multline*}

Note that for any $t\in [0,T]$ by \eqref{e:He}  we have
\[
\left| \int_{\O\times \O} \phi_{(1-\e_n)  \chi_{\e_n} \ast \rho_n + \e_n \orho}(x,x') \chi_{\e_n} \ast \rho_n^{\f'}(x') \cdot \chi_{\e_n} \ast ( \rho_n^{\n \f} \chi)(x) \dx' \dx \right| \leq C,
\]
uniformly in $n$. Therefore, for any subset $E \ss [0,T]$, 
\begin{equation}\label{e:Esmall}
\begin{split}
\left| \int_E \int_{\O\times \O} \phi_{(1-\e_n)  \chi_{\e_n} \ast \rho_n + \e_n \orho}(x,x') \chi_{\e_n} \ast \rho_n^{\f'}(x') \cdot \chi_{\e_n} \ast ( \rho_n^{\n \f} \chi)(x) \dx' \dx \dt \right| & \leq C |E|\\
\left| \int_E \int_{\O}  \int_\O \phi_{\rho}(x,x') \rho^{\f'}(x') \cdot \rho^{\n \f}(x)  \chi(t,x) \dx' \dx \dt \right| & \leq C |E|
\end{split}
\end{equation}

Since
\[
\|(1-\e_n)  \chi_{\e_n} \ast \rho_n + \e_n \orho - \rho \|_{L^1([0,T] \times \O} \to 0,
\]
by Egorov's Theorem, for any $\d_1>0$ there exists a subset $E_{\d_1} \ss [0,T]$ of measure $|E_{\d_1}| < \d_1$ such that 
\begin{equation}\label{e:Eunif}
\sup_{t\in [0,T] \backslash E_{\d_1}} \|(1-\e_n)  \chi_{\e_n} \ast \rho_n(t) + \e_n \orho(t) - \rho(t) \|_{L^1(\O)} \to 0.
\end{equation}
Now let us fix an arbitrary small $\d_2>0$, and consider the set $S_{\d_2} = \{ (t,x,y): \th_{\rho(t)}(x,y) > \d_2\}$.  In view of \ref{i:th3} and \eqref{e:Eunif} starting from some large $n>N$ we have
$\th_{(1-\e_n)  \chi_{\e_n} \ast \rho_n(t) + \e_n \orho(t)}(x,y) > \d_2 / 2$ for all $(t,x,x') \ss S_{\d_2} \cap (E^c_{\d_1} \times \O \times \O)$.  Then by \ref{i:b2} for all such points uniformly we have
\begin{equation}\label{e:kerlimit}
| \phi_{(1-\e_n)  \chi_{\e_n} \ast \rho_n + \e_n \orho}(x,x') - \phi_{\rho}(x,x') | = o(1).
\end{equation}

Let us now estimate the limit with $\d_1$ and $\d_2$ fixed and letting $n\to \infty$:
\begin{equation*}\label{}
\begin{split}
& \left| \int_0^T\int_{\O}  \int_\O \phi^n_{\rho_n}(x,x') \rho_n^{\f'}(x') \cdot \rho_n^{\n \f}(x) \chi(t,x) \dx' \dx \dt  -  \int_0^T\int_{\O}  \int_\O \phi_{\rho}(x,x') \rho^{\f'}(x') \cdot \rho^{\n \f}(x)  \chi(t,x) \dx' \dx \dt \right| \\
& \leq C \d_1+  \left| \int_{E^c_{\d_1}}\int_{\O}  \int_\O  \phi_{(1-\e_n)  \chi_{\e_n} \ast \rho_n + \e_n \orho}(x,x') \chi_{\e_n} \ast \rho_n^{\f'}(x') \cdot \chi_{\e_n} \ast ( \rho_n^{\n \f} \chi)(x) \dx' \dx \dt  \right. \\
&\hspace{3.1in} \left. -  \int_{E^c_{\d_1}}\int_{\O}  \int_\O \phi_{\rho}(x,x') \rho^{\f'}(x') \cdot \rho^{\n \f}(x)  \chi(t,x) \dx' \dx \dt \right|\\
& \leq C \d_1+  \left| \int_{S_{\d_2} \cap (E^c_{\d_1} \times \O \times \O)}  \phi_{(1-\e_n)  \chi_{\e_n} \ast \rho_n + \e_n \orho}(x,x') \chi_{\e_n} \ast \rho_n^{\f'}(x') \cdot \chi_{\e_n} \ast ( \rho_n^{\n \f} \chi)(x) \dx' \dx \dt  \right. \\
&\hspace{3.1in} \left. -  \int_{S_{\d_2} \cap (E^c_{\d_1} \times \O \times \O)} \phi_{\rho}(x,x') \rho^{\f'}(x') \cdot \rho^{\n \f}(x)  \chi(t,x) \dx' \dx \dt \right|\\
& +  \left| \int_{\th_\rho \leq \d_2}  \phi_{(1-\e_n)  \chi_{\e_n} \ast \rho_n + \e_n \orho}(x,x') \chi_{\e_n} \ast \rho_n^{\f'}(x') \cdot \chi_{\e_n} \ast ( \rho_n^{\n \f} \chi)(x) \dx' \dx \dt  \right. \\
&\hspace{3.1in} \left. -  \int_{\th_\rho \leq \d_2} \phi_{\rho}(x,x') \rho^{\f'}(x') \cdot \rho^{\n \f}(x)  \chi(t,x) \dx' \dx \dt \right|.
\end{split}
\end{equation*}
Let us examine the first difference:
\begin{equation*}\label{}
\begin{split}
 &\left| \int_{S_{\d_2} \cap (E^c_{\d_1} \times \O \times \O)}  \phi_{(1-\e_n)  \chi_{\e_n} \ast \rho_n + \e_n \orho}(x,x') \chi_{\e_n} \ast \rho_n^{\f'}(x') \cdot \chi_{\e_n} \ast ( \rho_n^{\n \f} \chi)(x) \dx' \dx \dt  \right. \\
&\hspace{3.1in} \left. -  \int_{S_{\d_2} \cap (E^c_{\d_1} \times \O \times \O)} \phi_{\rho}(x,x') \rho^{\f'}(x') \cdot \rho^{\n \f}(x)  \chi(t,x) \dx' \dx \dt \right|\\
\leq &  \left| \int_{S_{\d_2} \cap (E^c_{\d_1} \times \O \times \O)} [ \phi_{(1-\e_n)  \chi_{\e_n} \ast \rho_n + \e_n \orho}(x,x') - \phi_\rho(x,x')] \chi_{\e_n} \ast \rho_n^{\f'}(x') \cdot \chi_{\e_n} \ast ( \rho_n^{\n \f} \chi)(x) \dx' \dx \dt  \right| \\
+ & \left| \int_{S_{\d_2} \cap (E^c_{\d_1} \times \O \times \O)} \phi_{\rho}(x,x')[\chi_{\e_n} \ast \rho_n^{\f'}(x') \cdot \chi_{\e_n} \ast ( \rho_n^{\n \f} \chi)(x) -  \rho^{\f'}(x') \cdot \rho^{\n \f}(x)  \chi(t,x) ]\dx' \dx \dt \right| \leq o(1),
\end{split}
\end{equation*}
by \eqref{e:kerlimit}, \ref{i:b1}, and \eqref{e:macroLp}.

Now let us estimate the integrals over ${\th_\rho \leq \d_2}$. We in fact estimate the integrals individually using the Hilber-Schmidt bounds \ref{i:f1} and \eqref{e:He}:
\begin{equation*}\label{}
\begin{split}
& \left| \int_{\th_\rho \leq \d_2}  \phi_{(1-\e_n)  \chi_{\e_n} \ast \rho_n + \e_n \orho}(x,x') \chi_{\e_n} \ast \rho_n^{\f'}(x') \cdot \chi_{\e_n} \ast ( \rho_n^{\n \f} \chi)(x) \dx' \dx \dt  \right. \\
&\hspace{3.1in} \left. -  \int_{\th_\rho \leq \d_2} \phi_{\rho}(x,x') \rho^{\f'}(x') \cdot \rho^{\n \f}(x)  \chi(t,x) \dx' \dx \dt \right| \\
\lesssim & \sqrt{\oH}\left( \int_{\th_\rho \leq \d_2} \chi_{\e_n} \ast \rho_n^{\f'}(x') \cdot \chi_{\e_n} \ast ( \rho_n^{\n \f} \chi)(x) \dx' \dx \dt \right)^{1/2} + \sqrt{\oH}\left( \int_{\th_\rho \leq \d_2} \rho^{\f'}(x') \cdot \rho^{\n \f}(x) \dx' \dx \dt \right)^{1/2}  \\
\lesssim & o(1) + \left( \int_{\th_\rho \leq \d_2} \rho^{\f'}(x') \cdot \rho^{\n \f}(x) \dx' \dx \dt \right)^{1/2} \lesssim o(1) + \left(  \int_{\th_\rho \leq \d_2} \rho(x') \cdot \rho(x) \dx' \dx \dt  \right)^{1/2} 
\end{split}
\end{equation*}
Since by \ref{i:th1} we have $\rho\otimes \rho( \th_\rho = 0) = 0$, then we can pick $\d_2>0$ so that $ \int_{\th_\rho \leq \d_2} \rho(x') \cdot \rho(x) \dx' \dx \dt $ is arbitrarily small. This concludes the proof.

\section{Applications}\label{s:applications}

In this section we will vastly expand the list of examples of bounded protocols building upon the classical ones presented in \sect{s:examples}, and others introduced in \cite{Shv-EA}. Once all the conditions of \defin{d:protocol} are checked, the results proved above apply to the corresponding Vlasov-alignment model based upon such protocol.

We start by introducing a family of protocols that naturally encompass the classical \ref{CS}, \ref{MT},  and  \ref{MT*}.

\subsection{$\g$-protocols}  

Let us fix a local defining kernel $\phi \in L^\infty(\O)$ with
\begin{equation}\label{e:floc}
\phi(x) \geq \l \one_{|x|\leq r}.
\end{equation}
To make a proper sense of the convolutions $\rho \ast \phi$ used below, we will assume that the family of admissible densities is given by $\cF_q = L^1_q(\O) \cap \cP(\O)$ if $\phi \in L^\infty(\O)$, and $\cF_q = \cP_q(\O)$ if $\phi \in C_b(\O)$, for some $q\geq 0$ (the use of higher weights $q>0$ will only be necessary in the unbounded settings). Define
\begin{equation}\label{Cg}
\begin{split}
\phi_\rho(x,y) & =  \frac{\phi(x-y)}{(\rho \ast \phi(x))^{\g_1}(\rho \ast \phi(y))^{\g_2}}, \quad  \g_1,\g_2, \g_1+ \g_2 \leq 1,\\
 \th_\rho(x,y) & = (\rho \ast \phi(x))^{\epsilon_1}  (\rho \ast \phi(y))^{\epsilon_2}, \quad \epsilon_i = 1 \text{ if } \g_i>0, \text{ and } \epsilon_i=0 \text{ if } \g_i\leq 0. 
\end{split} \tag{$\cC_{\g_1,\g_2}$}
\end{equation}
Thus, $\cC_{0,0}$ = \ref{CS},  $\cC_{1,0}$ = \ref{MT},  $\cC_{0,1}$ = \ref{MT*}, see \fig{f:gamma}.
Condition \ref{i:th1} is verified due to the locality. Indeed, if $\epsilon_1 = \epsilon_2 = 1$, then any point where $\th_\rho(x,y) = 0$ will have $\rho(B_r(x) ) \rho(B_r(y)) = 0$ and hence the set $\{\th_\rho = 0\}$ will not intersect the topological support of $\rho \otimes \rho$. If say $\epsilon_1=1$ and $\epsilon_2=0$, then the same conclusion holds for $\supp \rho_x$. If $\epsilon_1 = \epsilon_2 = 0$, then the set $\{\th_\rho = 0\}$ is empty. Condition \ref{i:th2} obviously holds for $R'  = R+r$.

\begin{figure}
	\centering
	\includegraphics[width=2in]{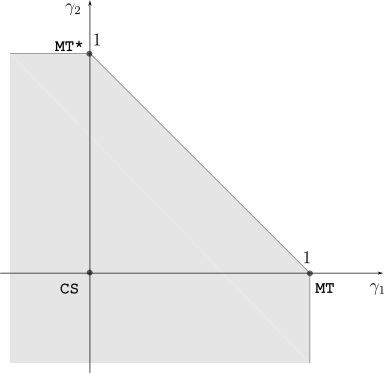}
	\caption{The allowed range of $(\g_1,\g_2)$ in $\gamma$-protocols.}
	\label{f:gamma}
\end{figure}

The communication operator is given by
\begin{equation}\label{ }
\bc_\rho(u) = \frac{1}{ (\rho \ast \phi)^{\g_1} }\left( \frac{u \rho}{ (\rho \ast \phi)^{\g_2} } \right)\ast \phi,
\end{equation}
and the strengths are given by
\begin{equation}\label{ }
\st_\rho =  \frac{1}{ (\rho \ast \phi)^{\g_1} }\left( \frac{\rho}{ (\rho \ast \phi)^{\g_2} } \right)\ast \phi, \quad \st_\rho^* =  \frac{1}{ (\rho \ast \phi)^{\g_2} }\left( \frac{\rho}{ (\rho \ast \phi)^{\g_1} } \right)\ast \phi.
\end{equation}

Proving  conditions \ref{i:f1}-\ref{i:f3} on the torus $\O = \T^n$ reduces to the following lemma, proved in \cite{KMT2013} for the case of $\g =1$, and for general $\g$ in \cite{Shv-EA}.

\begin{lemma}\label{} Let $\O = \T^n$.
For any $0 \leq \g \leq 1$ there exists $C>0$ such that for all $\rho \in \cF_0(\T^n)$ one has 
\begin{align}
\left( \frac{\rho}{(\rho \ast \phi)^\g} \right) \ast \phi(x) & \leq C (\rho \ast \phi(x))^{1-\g}, \label{e:KMTb}\\
 \int_{\O} \frac{\drho(x)}{\rho\ast \phi(x) } & \leq C.\label{e:KMT}
\end{align}
Consequently, \ref{Cg} defines a bounded  communication protocol on $\cF_0(\T^n)$.
\end{lemma}
\begin{proof}
To see \ref{i:f2}, if $0\leq \g_2 \leq 1$, then from \eqref{e:KMTb} we obtain
\[
\st_\rho(x) \leq C  (\rho \ast \phi)^{1-\g_1 - \g_2} \leq C'.
\]
If $\g_2<0$, then we can ignore the denominator and obtain
\[
\st_\rho(x) \leq C  (\rho \ast \phi)^{1 - \g_1} \leq C'.
\]
For the adjoint function the argument is similar.

The Hilbert-Schmidt condition \ref{i:f1} follows from \eqref{e:KMT}. Indeed,
\begin{equation*}\label{}
\begin{split}
 \int_{\O \times \O} \phi^2_\rho(x,y) \drho(y) \drho(x) \leq \|\phi\|_\infty \int_{\O \times \O}  \frac{\phi(x-y)}{(\rho \ast \phi(x))^{2\g_1}(\rho \ast \phi(y))^{2\g_2}} \drho(y) \drho(x).
\end{split}
\end{equation*}
Since $\g_1+ \g_2 \leq 1$ one of the exponents must not exceed $1/2$. Let us assume for definiteness that $\g_2 \leq \frac12$, or $2\g_2 \leq 1$. If also $\g_2 \geq 0$, then  \eqref{e:KMT}  applies to show 
\[
\lesssim \int_\O (\rho \ast \phi(x))^{1-2\g_1-2\g_2} \drho(x) \lesssim \int_{\O} \frac{\drho(x)}{\rho\ast \phi(x) } \leq C.
\]
If $\g_2 <0$, then we can simply ignore the $\g_2$-term, and noting that $2\g_1 -1 \leq 1$, we obtain
\[
\lesssim \int_\O (\rho \ast \phi(x))^{1-2\g_1} \drho(x) \lesssim \int_{\O} \frac{\drho(x)}{\rho\ast \phi(x) } \leq C.
\]
\end{proof}

Conditions \ref{i:b1} - \ref{i:b2} are trivial since positivity of thickness makes the denominator positive as well.

In the open space the verification of \ref{i:f1} - \ref{i:f3} is more subtle due to long range interactions. In fact, \eqref{e:KMTb} was proved originally in \cite{KMT2013} on $\O = \R^n$ for the case $\g=1$ and local kernels with compact support \eqref{e:KMTrR}.
Here, $C \sim (R/r)^n$. The condition \eqref{e:KMTrR} prohibits kernels with fat tails which is the main class for the application of the Cucker-Smale theorem. Nonetheless, for power-like kernels the protocol becomes bounded in the class of densities in $\cF_q$ for any $q$ large enough.

\begin{lemma}\label{l:KMTopen} Let $\O = \R^n$. Suppose for some $\l, \L, \b >0$, a kernel $\phi \in L^\infty(\R^n)$ satisfies
\begin{equation}\label{ }
\frac{\l}{\jap{x}^\b} \leq \phi(x) \leq \frac{\L}{\jap{x}^\b}.
\end{equation}
Then for any $q \geq \b$, $0\leq \g \leq 1$, and $\g_1,\g_2$ as above, we have
\begin{align}
\left( \frac{\rho}{(\rho \ast \phi)^\g} \right) \ast \phi(x) & \leq C \|\rho\|_{\cP_q}^{2 \b/q} (\rho \ast \phi(x))^{1-\g} , \label{e:KMTF}\\
\int_{\O \times \O} \phi_\rho^2(x,y) \drho(y) \drho(x) & \leq C\|\rho\|_{\cP_q}^{4\b/q}.\label{e:KMTHS}
\end{align}
Consequently, \ref{Cg} defines a bounded  communication protocol on $\cF_q(\R^n)$.
\end{lemma}
\begin{proof}
The proof of \ref{i:f2}-\ref{i:f3} goes as before provided we establish \eqref{e:KMTb}. So, let us focus on the latter. 

We first observe the exact asymptotic behavior of the convolution:
\begin{equation}\label{e:asymrhof}
 \frac{c \|\rho\|_{\cF_q}^{-\b/q}}{\jap{x}^\b} \leq \rho \ast \phi(x) \leq \frac{C \|\rho\|^{\b/q}_{\cF_q}}{\jap{x}^\b}.
\end{equation}
Indeed, from above, we have
\[
\rho \ast \phi(x)  \leq \L \int_{\R^n} \frac{\drho(y)}{\jap{x-y}^\b}.
\]
Let us split the integral into three regions: $|y| <\frac12 |x|$, $|y| > 2 |x|$, and $\frac12|x| \leq |y| \leq 2|x|$. We have
\begin{equation*}\label{}
\begin{split}
\int_{|y| <\frac12 |x|} \frac{\drho(y)}{\jap{x-y}^\b} & \lesssim \frac{1}{\jap{x}^\b} \\
\int_{|y| > 2 |x|} \frac{\drho(y)}{\jap{x-y}^\b} & \lesssim \int_{|y| > 2 |x|} \frac{\drho(y)}{\jap{y}^\b} \lesssim \frac{1}{\jap{x}^\b} \int_{|y| > 2 |x|} \drho(y) \leq  \frac{1}{\jap{x}^\b} \\
\int_{ \frac12|x| \leq |y| \leq 2|x| } \frac{\drho(y)}{\jap{x-y}^\b} & \leq \int_{ \frac12|x| <|y|<2|x| } \drho(y) \lesssim  \frac{1}{\jap{x}^\b} \int_{ \frac12|x| <|y|<2|x| } \jap{y}^\b \drho(y) \leq \frac{\|\rho\|^{\b/q}_{\cP_q}}{\jap{x}^\b} .
\end{split}
\end{equation*}
Now, from below we have
\[
\rho \ast \phi(x)  \geq \l \int_{\R^n} \frac{\drho(y)}{\jap{x-y}^\b}.
\]
Note that for $R = (2 \|\rho\|_{\cF_q})^{1/q}$ we have $\int_{|y| >R} \drho(y)  < \frac12$. So, the measure is substantially concentrated in the ball of radius $R$: $\int_{|y|  \leq R} \drho(y)  \geq  \frac12$. So, now if $|x| > 2R$, then
\[
\int_{\R^n} \frac{\drho(y)}{\jap{x-y}^\b} \geq \int_{|y|  \leq R} \frac{\drho(y)}{\jap{x-y}^\b} \gtrsim  \frac{1}{\jap{x}^\b} \int_{|y| \leq R} \drho(y) \geq \frac{1}{2\jap{x}^\b}.
\]
But if $|x| \leq 2R$, then
\[
\int_{\R^n} \frac{\drho(y)}{\jap{x-y}^\b} \geq \frac{c}{R^\b} \int_{|y| \leq R} \drho(y) \geq \frac{c}{2R^\b} \geq \frac{c}{2R^\b} \frac{1}{\jap{x}^\b} \gtrsim   \|\rho\|_{\cP_q}^{-\b/q}\frac{1}{\jap{x}^\b} .
\]
This proves \eqref{e:asymrhof}.

To prove \eqref{e:KMTb} we use the established asymptotics:
\[
\left( \frac{\rho}{(\rho \ast \phi)^\g} \right) \ast \phi(x)  \lesssim \|\rho\|_{\cF_q}^{\g \b/q} \int_{\R^n} \frac{\jap{y}^{\b \g} \drho(y)}{\jap{x-y}^\b}.
\]
Let us split the integral into three regions as before: $|y| <\frac12 |x|$, $|y| > 2 |x|$, and $\frac12|x| \leq |y| \leq 2|x|$. We have
\begin{equation*}\label{}
\begin{split}
\int_{|y| <\frac12 |x|} \frac{ \jap{y}^{\b \g} \drho(y)}{\jap{x-y}^\b} & \lesssim \frac{1}{\jap{x}^{(1-\g)\b}} \\
\int_{|y| > 2 |x|} \frac{\jap{y}^{\b \g}\drho(y)}{\jap{x-y}^\b} & \lesssim \int_{|y| > 2 |x|} \frac{\drho(y)}{\jap{y}^{(1-\g)\b}} \leq \frac{1}{\jap{x}^{(1-\g)\b}} \\
\int_{ \frac12|x| \leq |y| \leq 2|x| } \frac{\jap{y}^{\b \g} \drho(y)}{\jap{x-y}^\b} & \leq \jap{x}^{\b \g}\int_{ \frac12|x| <|y|<2|x| } \drho(y) \lesssim  \frac{1}{\jap{x}^{(1-\g)\b}} \int_{ \frac12|x| <|y|<2|x| } \jap{y}^\b \drho(y) \leq \frac{\|\rho\|^{\b/q}_{\cP_q}}{\jap{x}^{(1-\g)\b}} .
\end{split}
\end{equation*}
So, we obtain
\[
\left( \frac{\rho}{(\rho \ast \phi)^\g} \right) \ast \phi(x)  \lesssim \frac{\|\rho\|_{\cP_q}^{(1+\g) \b/q} } {\jap{x}^{(1-\g)\b}}  \lesssim ( \rho \ast \phi(x) )^{1-\g} \|\rho\|_{\cP_q}^{2 \b/q}.
\]

Finally, as to \eqref{e:KMTHS}, we can assume that both $\g_i \geq 0$, for otherwise we can ignore the corresponding term $\rho \ast \phi$ and just assume that $\g_i = 0$. We have
\begin{equation*}\label{}
\begin{split}
\int_{\O \times \O}  \frac{\phi^2(x-y)}{(\rho \ast \phi(x))^{2\g_1}(\rho \ast \phi(y))^{2\g_2}} \drho(y) \drho(x) \lesssim  \|\rho\|_{\cP_q}^{2(\g_1+ \g_2)\b/q} \int_{\O \times \O}  \frac{\jap{x}^{2\g_1 \b}\jap{y}^{2\g_2 \b}}{\jap{x-y}^{2\b}} \drho(y) \drho(x)
\end{split}
\end{equation*}
Splitting the integral into three regions as before, we obtain since $\g_1+ \g_2 \leq 1$,
\begin{equation*}\label{}
\begin{split}
\int_{|y| < \frac12 |x| }  \frac{\jap{x}^{2\g_1 \b}\jap{y}^{2\g_2 \b}}{\jap{x-y}^{2\b}} \drho(y) \drho(x) &\lesssim \int_{|y| < \frac12 |x| }  \frac{\jap{x}^{2(\g_1+\g_2) \b}}{\jap{x}^{2\b}} \drho(y) \drho(x)  \leq 1 \\
\int_{|x| < \frac12 |y| }  \frac{\jap{x}^{2\g_1 \b}\jap{y}^{2\g_2 \b}}{\jap{x-y}^{2\b}} \drho(y) \drho(x) &\lesssim \int_{|x| < \frac12 |y| }  \frac{\jap{y}^{2(\g_1+\g_2) \b}}{\jap{y}^{2\b}} \drho(y) \drho(x)  \leq 1 \\
\int_{\frac12|x| \leq |y| \leq 2 |x| }  \frac{\jap{x}^{2\g_1 \b}\jap{y}^{2\g_2 \b}}{\jap{x-y}^{2\b}} \drho(y) \drho(x) & \leq \int_{\frac12|x| \leq |y| \leq 2 |x| } \jap{x}^{2(\g_1+\g_2) \b} \drho(y) \drho(x) \\
&\leq \int_{\O\times \O} \jap{y}^{(\g_1+\g_2) \b} \drho(y) \jap{x}^{(\g_1+\g_2) \b} \drho(x) \\
& \leq  \|\rho\|_{\cP_q}^{2\b/q}.
\end{split}
\end{equation*}
We arrive at 
\[
\int_{\O \times \O}  \frac{\phi^2(x-y)}{(\rho \ast \phi(x))^{2\g_1}(\rho \ast \phi(y))^{2\g_2}} \drho(y) \drho(x) \lesssim \|\rho\|_{\cP_q}^{4\b/q}.
\]

\end{proof}

\subsection{Bochner-positive  $\g$-protocols}  Another way to symmetrize the Motsch-Tadmor model is through a Bochner-type representation, which also makes the communication  positive semi-definite. This model was first proposed in \cite{Shv-hypo}. We extend it to arbitrary scaling on $\rho$. 

We adopt the model to the periodic and open space cases differently. Let us consider $\O = \T^n$ first.

Let $\g \leq 1$, and let $\psi\in C(\O)$, $\psi \geq 0$, be an even kernel $\psi(x) = \psi(-x)$. Let $\cF = L^1(\O) \cap \cP(\O)$. Consider
\begin{equation}\label{Cg+T}\tag{$\cC_{\g}^+(\T^n)$}
\begin{split}
\phi_\rho(x,y) &= \int_{\O} \frac{\psi(x-z)\psi(y-z)}{(\rho \ast \psi(z))^{\g}}\dz, \\
 \th_\rho(x,y)  & = \begin{cases} \inf_{z \in (x+ \supp \psi) \cap (y + \supp \psi)} \rho \ast \psi(z), \text{ if }  (x + \supp \psi) \cap (y  + \supp \psi) \neq \varnothing \\
\|\psi\|_\infty, \text{ if }  (x +  \supp \psi) \cap (y + \supp \psi) = \varnothing. \end{cases}
\end{split}
\end{equation}
The corresponding communication operator is given by the overmollified Favre filtration 
\begin{equation}\label{ }
\bc_\rho(u) =\left(\frac{(u\rho)\ast \psi}{(\rho \ast \psi)^\g}\right)\ast \psi, \quad \st_\rho = ((\rho \ast \psi)^{1-\g})\ast \psi,
\end{equation}
and it is positive semi-definite
\[
(\bc_\rho(u),u)_\rho = \int_\O \frac{((u\rho) \ast \psi)^2}{(\rho \ast \psi)^\g} \dx \geq 0.
\]
Clearly, \ref{i:f2} - \ref{i:f3} hold on any environment due to $\psi \in L^\infty$. As to the Hilbert-Schmidt property  \ref{i:f1},  
\begin{equation*}\label{}
\begin{split}
\int_{\O \times \O} \phi^2_\rho(x,y) \drho(y) \drho(x) & \leq \int_{\O \times \O} \int_{\O} \frac{\psi(x-z)\psi(y-z)}{(\rho \ast \psi(z))^{2 \g}}\dz \drho(y) \drho(x) \\
&\leq  \int_{\O} (\rho \ast \psi(z))^{2-2 \g}\dz  \leq \|\psi\|_\infty^{2-2\g} |\O|.
\end{split}
\end{equation*}
Conditions \ref{i:b1} - \ref{i:b2} are also trivial to verify.

The reason why we are not allowed to use general probability measure flocks $\cP$ in this protocol (unless of course $\supp \psi = \O$) is because of the failure of \ref{i:th1}.
\begin{example}
Suppose the support of $\psi$ is exactly the unit ball $B_1(0)$ and $\psi(x)>0$ for $|x|<1$. Let $\rho = \d_0$. We have $\th_\rho(0,0) = \inf_{z \in \supp \psi}\psi(z) = 0$. So, in this case $\rho \otimes \rho ( \{\th_\rho = 0 \} ) > 0$. 
\end{example}

To verify axioms of thickness, we note that \ref{i:th2} and \ref{i:th3} are trivial. Let us note that the map $x \to x + \supp \psi$ is upper semi-continuous, i.e. for every open $x_0 + \supp \psi \ss U$ there exists open $x_0 \in V$ such that $y + \supp \psi \ss U$ for all $y\in V$. To the check lower semi-continuity of $\th_\rho$ let us fix a point $(x,y)$ such that $(x + \supp \psi) \cap (y + \supp \psi) \neq \varnothing$ (otherwise there is nothing to check). Call $Z_{x,y} = (x + \supp \psi) \cap (y + \supp \psi)$. Then $\th_\rho(x,y)  = \inf_{z\in Z_{x,y}} \rho \ast \psi(z)$. For any $\e>0$, by continuity, there exists an open nightborhood $U \supset Z_{x,y}$ such that $\inf_{z \in U} \rho \ast \psi(z) > \th_\rho(x,y) - \e$. Note that the sets $(x + \supp \psi) \backslash U$, $(y + \supp \psi )\backslash U$ are closed and disjoint. So, there exist two open disjoint neighborhoods $U_x \supset x +\supp \psi$ and $U_y \supset y + \supp \psi$, $U_x \cap U_y = \emptyset$. Note that $U \cup U_x$ covers $x + \supp \psi$ while $U \cup U_y$ covers $y + \supp \psi$. By upper semicontinuity observed earlier, there exist two open neighborhoods $x\in V_x$, $y \in V_y$ such that for all $\tilde{x} \in V_x$, $\tilde{y} \in V_y$ we have $\tilde{x} + \supp \psi \ss U \cup U_x$ and  $\tilde{y} + \supp \psi \ss U \cup U_y$. Since $U_x$ and $U_y$ are disjoint, we have $(\tilde{x} + \supp \psi ) \cap (\tilde{y} + \supp \psi ) \ss U$, and hence, $\inf_{z \in (\tilde{x} + \supp \psi ) \cap (\tilde{y} + \supp \psi )} \rho \ast \psi(z) > \th_\rho(x,y) - \e$. This proves that $\liminf_{(x',y') \to (x,y)} \th_\rho(x',y') \geq \th_\rho(x,y)$.

It remains to check \ref{i:th1}.  For that property to hold we need to assume that the support of $\psi$ retains a positive capacity at the boundary:
\begin{equation}\label{e:cap}
\exists r_0>0, \ 0<\a<1 \text{ such that } \forall x\in \p \supp \psi, \ r<r_0:  \quad |B_r(x) \cap \{\psi>0\}| \geq \a |B_r(x)|.
\end{equation}

 Let us fix $\rho \in L^1(\O)$. Let us assume on the contrary that $\rho \otimes\rho ( \{\th_\rho = 0 \} ) > 0$. Then there exists a $\e>0$ such that $\int_{\rho \geq \e}\int_{\rho \geq \e} \one_{\{\th_\rho = 0\}}(x,y) \rho(y) \dy \rho(x) \dx >0$.  Let us fix a $\d \ll \e$ and consider the sum
\[
\int_{\rho \geq \e} \one_{\{\th_\rho = 0\}}(x) \rho(x) \dx = \sum_{l=0}^\infty \sum_{k=0}^\infty \int_{\e + (k+1) \d \geq \rho \geq \e + k \d} \int_{\e + (l+1) \d \geq \rho \geq \e + l \d}  \one_{\{\th_\rho = 0\}}(x,y) \rho(y) \dy \rho(x) \dx >0.
\]
So, there will exist a pair $k , l \geq 0$ such that 
\[
 \int_{\e + (k+1) \d \geq \rho \geq \e + k \d} \int_{\e + (l+1) \d \geq \rho \geq \e + l \d}  \one_{\{\th_\rho = 0\}}(x,y) \rho(y) \dy \rho(x) \dx >0.
\]
So, the Lebesgue measure of the set where the function
\[
g(x,y) = \one_{\{\e + (k+1) \d \geq \rho \geq \e + k \d\}} (x) \one_{\{\e + (l +1) \d \geq \rho \geq \e + l \d\}} (y) \one_{\{\th_\rho = 0\}}(x,y) \rho(x) \rho(y)
\]
is positive is positive. Applying the Lebesgue Differentiation Theorem to the function 
we find a point $(x,y)$ (in fact almost all points on the support) such that 
\begin{equation}\label{e:g>0}
\lim_{r \to 0} \frac{1}{|B_r(x)| |B_r(y)|} \int_{B_r(x) \times B_r(y)} g(x',y') \dx' \dy'  = g(x,y) >0.
\end{equation}
In particular, $\th_\rho(x,y) = 0$, and $\rho(x) \geq \e + k \d$, $\rho(y) \geq \e + l \d$. So,
\[
\lim_{r \to 0} \frac{1}{|B_r(x)| |B_r(y)|} \int_{B_r(x) \times B_r(y)} g(x',y') \dx' \dy'  \geq (\e + k \d)(\e+ l \d).
\]
Yet, since, $\th_\rho(x,y) = 0$, by continuity there exists a $z\in Z_{x,y}$  such that $\rho \ast \psi(z)  = 0$. If in fact $\psi(z-x)>0$, then by continuity $\psi(z-x') \geq  a>0$ for $x'\in B_r(x)$, and hence, $\rho(B_r(x)) = 0$ for some small $r$ and this leads immediately to a contradiction with \eqref{e:g>0}.  At the same time  $z - x \in \supp \psi$ by the choice of $z$. So, necessarily $z\in x+\p \supp \psi$.   Since $\rho \ast \psi(z)  = 0$, we have
\[
\int_{x' \in B_r(z-x) \cap \{\psi>0\}} \rho(z-x')\psi(x') \dx' = 0,
\]
for all $r>0$. Thus, $\rho(z -  B_r(z-x) \cap \{\psi>0\}) = 0$. Since $z- x\in \p \supp \psi$, by our capacity assumption on the support, $|B_r(z-x) \cap \{\psi>0\}| \geq \a |B_r|$. Since $z -  B_r(z-x) \cap \{\psi>0\} = B_r(x) \cap (z + \supp \psi)$, we found a subset  of the ball $B_r(x)$, namely $A_r(x) = B_r(x) \cap (z + \supp \psi)$ of measure at least $\a |B_r(x)|$ on which $\rho = 0$ a.e.
Consequently, we have
\begin{equation*}\label{}
\begin{split}
(\e + k \d)(\e+ l \d) & \leq \lim_{r \to 0} \frac{1}{|B_r(x)| |B_r(y)|} \int_{B_r(x) \times B_r(y)} g(x',y') \dx' \dy'  \\
& =\lim_{r \to 0} \frac{1}{|B_r(x)| |B_r(y)|} \int_{B_r(x) \backslash A_r(x)} \int_{B_r(y)} g(x',y') \dx' \dy' \\
& \leq   \lim_{r \to 0} \frac{(\e + (k+1) \d)(\e + (l+1) \d)}{|B_r(x)|} (1-\a) |B_r(x)|\\
& \leq (1-\a) (\e + (k+1) \d)(\e + (l+1) \d).
\end{split}
\end{equation*}
This is a contradiction if $\d$ is chosen small enough.

\begin{lemma}\label{}
 \ref{Cg+T} defines a bounded protocol on  $\cF = L^1(\T^n) \cap \cP(\T^n)$ for any even kernel $\psi\in C(\O)$, $\psi \geq 0$ satisfying the capacity condition \eqref{e:cap}.
\end{lemma}

In the open space, $\O = \R^n$, we assume  $\psi \in L^\infty(\R^n)$ satisfies
\begin{equation}\label{e:psib}
\frac{\l}{\jap{x}^\b} \leq \psi(x) \leq \frac{\L}{\jap{x}^\b}, \text{ for some } \b >n.
\end{equation}
Let $\cF = \cF_q$ be as chosen for $\g$-protocols, but with $q \geq 2\b$. Define the protocol by
\begin{equation}\label{Cg+R}\tag{$\cC_{\g}^+(\R^n)$}
\begin{split}
\phi_\rho(x,y) &= \int_{\O} \frac{\psi(x-z)\psi(y-z)}{(\rho \ast \psi(z))^{\g}}\dz, \\
 \th_\rho(x,y)  & = \frac{1}{\jap{x}^\b \vee \jap{y}^\b}.
\end{split}
\end{equation}
The choice of thickness will be clear in a moment, after we estimate the kernel. 

By \eqref{e:asymrhof}, 
\begin{equation}\label{e:phiopen}
\phi_\rho(x,y) \leq C \|\rho\|^{\g \b/q}_{\cF_q} \int_{\R^n} \frac{\jap{z}^{\g \b}}{\jap{x-z}^\b \jap{y-z}^\b} \dz.
\end{equation}
Note that since $\b>n$ and $\g \leq 1$, the integral converges. We now prove the following estimate 
\begin{equation}\label{ }
\phi_\rho(x,y) \leq C \|\rho\|^{\g \b/q}_{\cF_q}\jap{x}^\b \vee \jap{y}^\b.
\end{equation}
Indeed, by symmetry, we can assume that $|x| \leq |y|$. Then splitting the integral we obtain
\begin{equation*}\label{}
\begin{split}
\int_{\R^n} \frac{\jap{z}^{\g \b}}{\jap{x-z}^\b \jap{y-z}^\b} \dz & = \int_{|z|<2|y|} \frac{\jap{z}^{\g \b}}{\jap{x-z}^\b \jap{y-z}^\b} \dz + \int_{|z| \geq 2|y|} \frac{\jap{z}^{\g \b}}{\jap{x-z}^\b \jap{y-z}^\b} \dz \\
& \lesssim \jap{y}^{\g\b} \int_{\R^n} \frac{\dz}{\jap{z}^\b} +  \int_{|z| \geq 2|y|}\frac{ \jap{z}^{\g \b} }{\jap{z}^{2\b}}\dz \\
& \lesssim \jap{y}^{\b} + \jap{y}^{\g\b+n - 2\b} \lesssim \jap{y}^{\b}.
\end{split}
\end{equation*}
Plugging this back into \eqref{e:phiopen} we obtain
\[
\int_{\O \times \O} \phi^2_\rho(x,y) \drho(y) \drho(x) \leq C \|\rho\|^{2\g \b/q}_{\cF_q}  \int_{\O \times \O} (\jap{x}^{2\b} + \jap{y}^{2\b} )\drho(y) \drho(x) \leq C \|\rho\|^{2(1+\g) \b/q}_{\cF_q}.
\]

\begin{lemma}\label{}
 \ref{Cg+R} defines a bounded protocol on  $\cF_q$ for $q \geq 2\b$ and $\b > n$, with $\psi \in L^\infty$ satisfying \eqref{e:psib}. 
\end{lemma}

\subsection{Segregation protocols} 

Another distinctly different and non-Galilean invariant protocol was introduced in \cite{Shv-EA} as a consensus model in segregated communities.  Let $\O = \T^n$. Consider any smooth partition of unity $g_l \in C(\O)$, $g_l \geq 0$, and $\sum_{l=1}^L g_l = 1$ subordinated to an open cover $\{ \cO_l \}_{l=1}^L$ of $\O$,  so that $\supp g_l \ss \cO_l$. We assume that $|\p \supp g_l | = 0$ for all $l$. Let in this case $\cF = L^1(\O)\cap \cP(\O)$. For a density $\rho\in \cP(\O)$ we denote $\rho(g_l) = \int_\O g_l \drho$. One defines the protocol by   
\begin{equation}\label{Segg}  \tag{Seg$_\g^+$}
\begin{split}
\phi_\rho(x,y) & =  \sum_{l = 1}^L \frac{ g_l(x) g_l(y)}{(\rho(g_l))^\g},\\
  \th_\rho(x,y) & = \begin{cases} \min_{ l:  x,y \in \supp g_l } \rho(g_l), \text{ if }  \exists l: x,y \in \supp g_l \\
 \|g_\cdot(\cdot)\|_\infty,  \text{ otherwise } . \end{cases}
 \end{split} 
\end{equation}
The superscript '$+$' designate that the model is positive semi-definite too.  The protocol is symmetric and the strength is given by $ \st_\rho(x) = \sum_{l=1 }^L g_l(x) \rho^{1-\g}(g_l)$. The conditions \ref{i:f2} - \ref{i:f3} are trivial. As to the Hilbert-Schmidt condition \ref{i:f1},
\[
\int_{\O \times \O} \phi^2_\rho(x,y) \drho(y) \drho(x) \leq C  \sum_{l = 1}^L \int_{\O \times \O} \frac{ g_l(x) g_l(y)}{(\rho(g_l))^{2\g}}\drho(y) \drho(x)  = C  \sum_{l = 1}^L \frac{ \rho(g_l)^2}{(\rho(g_l))^{2\g}} \leq C'.
\]

The only non-trivial properties to check here are the lower semi-continuity and \ref{i:th1}.  Let us fix a pair $(x,y)$. If $x$, $y$ or both are outside of any single $\supp g_l$, then due their finite number, then there exist open balls $B_r(x)$, $B_r(y)$ that fulfill the same property. Hence, $\th_\rho(B_r(x),B_r(y)) =  \|g_\cdot(\cdot)\|_\infty$. If $x,y \in \supp g_l$ for some $l$, let us denote by $k_1,...,k_p$ the indexes $k$ for which either $x\not \in \supp g_{k}$ or  $y \not \in \supp g_{k}$. Since the supports are compact, there is $\e>0$ such that with $B_\e(x) \cap \supp g_k = \emptyset$ or $B_\e(y) \cap \supp g_k = \emptyset$. This means that if $(x',y') \in B_\e(x) \times B_\e(y)$, then the pair $x',y'$ cannot belong to  $\supp g_{k_i}$ for any $i$. So, the collection of indexes for which $x',y' \in  \supp g_{l}$ will be smaller than or equal to that of the original pair $(x,y)$. Hence, the infimum will be at least that of $\th_\rho(x,y)$.

To show \ref{i:th1} suppose we have $\th_\rho(x,y) = 0$, hence there exists $l$ such that $x,y \in \supp g_l$ and $\rho(g_l)=0$. If $g_l(x) >0$, then $\rho(B_\e(x))=0$ for a small $\e$, hence $(x,y)\not\in \supp \rho \otimes \rho$. Similarly, if $g_l(y) >0$. Otherwise, $x, y\in \p (\supp g_l)$. So, 
\[
 \{ \th_\rho =0 \}  \ss( \O \times \O \backslash \supp \rho \otimes \rho ) \cup \p (\supp g_1) \times \p (\supp g_1) \cup \dots \cup \p (\supp g_L) \times \p (\supp g_L).
 \]
Since $\rho \in L^1$,  the $\rho \otimes \rho$-measure of the set on the right hand side is $0$.

\begin{lemma}\label{}
 \ref{Segg} defines a bounded protocol on  $\cF = L^1(\O)\cap \cP(\O)$, provided $g_l$'s satisfy the assumptions above. 
\end{lemma}


\end{document}